\documentclass[11pt, reqno]{amsart}
\usepackage[latin9]{inputenc}
\usepackage{enumitem}
\usepackage{amsthm}
\usepackage{amstext}
\usepackage{amssymb}
\usepackage{esint}
\usepackage{xcolor}

\makeatletter
\usepackage{enumitem}		% customizable list environments
\theoremstyle{plain}
\newtheorem{thm}{\protect\theoremname}[section]
\theoremstyle{plain}

\theoremstyle{plain}
\newtheorem{lem}[thm]{\protect\lemmaname}
\theoremstyle{plain}
\newtheorem{prop}[thm]{\protect\propositionname}
\theoremstyle{plain}
\newtheorem{defn}[thm]{\protect\definitionname}
\theoremstyle{remark}
\theoremstyle{plain}
\newtheorem{remark}[thm]{\protect\remarkname}
\theoremstyle{remark}
\newtheorem*{rem*}{\protect\remarkname}
 \newlist{casenv}{enumerate}{4}
 \setlist[casenv]{leftmargin=*,align=left,widest={iiii}}
 \setlist[casenv,1]{label={{\itshape\ \casename} \arabic*.},ref=\arabic*}
 \setlist[casenv,2]{label={{\itshape\ \casename} \roman*.},ref=\roman*}
 \setlist[casenv,3]{label={{\itshape\ \casename\ \alph*.}},ref=\alph*}
 \setlist[casenv,4]{label={{\itshape\ \casename} \arabic*.},ref=\arabic*}
\theoremstyle{plain}
\newtheorem{cor}[thm]{\protect\corollaryname}
 \usepackage{amsthm}\usepackage{amsfonts}\usepackage{latexsym}

\newcommand{\RR}{{\mathbb{R}}}

\newcommand{\NN}{\mathbb{N}}

\newcommand{\FF}{\mathcal{F}}
\newcommand{\OO}{\mathcal{O}}
\newcommand{\WW}{{W}}

\newcommand{\dd}{\text{\,d}}
\newcommand{\LL}{{L}}

\let\emptyset\varnothing

\providecommand{\assumptionname}{Assumption}
\providecommand{\corollaryname}{Corollary}
\providecommand{\definitionname}{Definition}
\providecommand{\lemmaname}{Lemma}
\providecommand{\remarkname}{Remark}
\providecommand{\casename}{Case}
\providecommand{\theoremname}{Theorem}
\providecommand{\propositionname}{Proposition}

\makeatother

  \providecommand{\assumptionname}{Assumption}
  \providecommand{\corollaryname}{Corollary}
  \providecommand{\definitionname}{Definition}
  \providecommand{\lemmaname}{Lemma}
  \providecommand{\remarkname}{Remark}
 \providecommand{\casename}{Case}
\providecommand{\theoremname}{Theorem}

\begin{document}

\title[Sobolev wavefront set and compactly supported wavelets]{
Wavelet characterizations of the Sobolev wavefront set: bandlimited wavelets and compactly supported wavelets}

\author{Hartmut F\"{u}hr, Mahya Ghandehari}

\address{H.~F\"uhr\\
Lehrstuhl f\"ur Geometrie und Analysis\\ 
RWTH Aachen University \\
 D-52056 Aachen\\
 Germany}
 
\address{M.~Ghandehari\\
Department of Mathematical Sciences\\ 
University of Delaware\\
Newark, DE 19716\\
 USA}
\begin{abstract}
We consider the problem of characterizing the Sobolev wavefront set of a tempered
distribution $u\in\mathcal{S}'(\mathbb{R}^{d})$ in terms of its continuous
wavelet transform, with the latter being defined with respect to a suitably
chosen dilation group $H\subset{\rm GL}(\mathbb{R}^{d})$.  We derive necessary and sufficient criteria for elements of the Sobolev wavefront set, formulated in terms of the decay behaviour of a given generalized continuous wavelet transform. 
It turns out that the characterization of directed smoothness of finite order can be performed in the two important cases: (1) bandlimited wavelets, and (2) wavelets with finitely many vanishing moments (e.g.~compactly supported wavelets).

The main results of this paper are based on a number of fairly technical conditions on the dilation group. In order to demonstrate their wide applicability, we exhibit a large class of generalized shearlet groups in arbitrary dimensions fulfilling all required conditions, and give estimates of the involved constants. 
\end{abstract}
\maketitle
\global\long\def\with{\,\middle|\,}

\hyphenation{an-iso-tropic}
\noindent \textbf{\small Keywords:}{\small {} wavefront set; square-integrable
group representation; continuous wavelet transform; anisotropic wavelet
systems; shearlets}{\small \par}

\noindent \textbf{\small AMS Subject Classification:}{\small {} 42C15;
42C40; 46F12}{\small \par}

\section{Introduction}
A recurring theme in wavelet analysis and its various generalizations is the relationship between smoothness of the analyzed function on the one hand, and the decay of its wavelet coefficients on the other. Global smoothness, as quantified by norms of smoothness spaces such as Besov spaces, is related to global decay properties, whereas local smoothness is related to a more localized decay behaviour of the coefficients. This paper is concerned with the later class of questions, for a large class of generalized continuous wavelet transforms.

Our aim is to give criteria for $N$-regular directed points of a tempered distribution $u\in {\mathcal S}'({\mathbb R}^d)$ in terms of continuous wavelet transform decay. A wavelet transform is defined with respect to a suitably
chosen dilation group $H\subset{\rm GL}(\mathbb{R}^{d})$ and a wavelet vector $\phi\in {\mathcal S}({\mathbb R}^d)$. We obtain such criteria for two important cases: (1) when the wavelet vector is compactly supported in frequency, and (2) when the wavelet vector has a certain minimal number of vanishing moments; this case allows to use wavelet vectors that are compactly supported.  

The $N$-regular directed points, formally defined in  equation \eqref{eqn:decay_cond_reg_dir} below, describe the oriented local regularity of a tempered distribution $u$: if $(x,\xi)$ is
an $N$-regular directed point of $u$, then $u$ can be considered smooth
of order $N$ at $x$, when viewed in direction $\xi$. The $N$-wavefront set $WF^N(u)$, a.k.a.~the Sobolev wavefront set of $u$ of order $N$, consists of all directed points $(x,\xi)$ that are not $N$-regular directed points of $u$. 
We typically describe the wavefront set $WF^N(u)$ by giving criteria for elements in its complement.

Ideally, we would like to achieve results of the following form: 
\begin{eqnarray}\label{eq:intro-char}
\quad(x,\xi)\mbox{ is an $N$-regular directed point of \ensuremath{u}} & \Leftrightarrow & \exists\mbox{ neighborhood }U\mbox{ of }x\ \ \forall\, y\in U\\\
 &  & \forall\, h\in K~:|W_{\psi}u(y,h)|\le C\|h\|^{N}~,\nonumber 
\end{eqnarray}
where $K\subset H$ is a suitable subset of dilations which explicitly
depends on $\psi$ and a certain (sufficiently small) cone containing
directions near $\xi$. Intuitively, $K$ contains those $h\in H$
such that $\pi(0,h)\psi$ is a small-scale wavelet oscillating in
direction $\xi$.
A characterization for wavefront sets of infinite order - similar to the one in \eqref{eq:intro-char} -  was given in \cite{FeFuVo}; here the decay of a fixed order $N$ on the right-hand side is replaced by decay of arbitrary degree. 
For $N$-regular directed points, we will settle for more modest statements, of the following types:
\begin{eqnarray} \label{eqn:main_goal_1}
& &(x,\xi)\mbox{ is an $N$-regular directed point of \ensuremath{u}} \\  
& \Rightarrow & \exists\mbox{ neighborhood }U\mbox{ of }x\ \forall\, y\in U
\ \forall\, h\in K_1~:|W_{\psi}u(y,h)|\le C\|h\|^{N-\delta_1} \nonumber ~,
\end{eqnarray}
typically called a ``direct'' theorem,
as well as an ``inverse'' theorem (both in the parlance of \cite{Grohs_2011}), namely
\noindent
\begin{eqnarray} \label{eqn:main_goal_2} 
& &\exists\mbox{ neighborhood }U\mbox{ of }x\  \forall\ y\in U\  \forall\, h\in K_2 :|W_{\psi}u(y,h)|\le C\|h\|^{\alpha N+\delta_2}\\ 
& \Rightarrow & (x,\xi)\mbox{ is an $N$-regular directed point of \ensuremath{u}} \nonumber 
\end{eqnarray}
with suitable constants $\alpha \ge 1, \delta_1,\delta_2 \ge 0$ and suitably defined subsets $K_1, K_2 \subset H$ of dilations. For the case of the shearlet transform in dimension two, a similar characterization had been obtained by Grohs \cite{Grohs_2011}. As related work, discussing this question with varying degrees of generality, we mention \cite{MR3501995,MR3591240,MR4047540}.
We point to the gap between the exponents of the direct and the inverse statements in \eqref{eqn:main_goal_1} and \eqref{eqn:main_goal_2}.
To our knowledge, this gap cannot be closed; see \cite{Grohs_2011} for results of similar type on shearlets.
Consequently, the results of our paper provide a \textit{near characterization} of $N$-regular directed points.
%
%with suitable constants $\alpha \ge 1, \delta_1,\delta_2 \ge 0$ and suitably defined subsets $K_1, K_2 \subset H$ of dilations. We point to the gap between the exponents of the direct and the inverse theorems.
%\pink{To our knowledge, this gap cannot be closed; see \cite{Grohs_2011} for similar results on shearlets.
%Consequently,} the results of our paper provide a \textit{near characterization} of $N$-regular directed points.
%
%For the case of the shearlet transform in dimension two, a similar characterization had been obtained
%by Grohs \cite{Grohs_2011}.  As related work, discussing this question with varying degrees of generality, we mention \cite{MR3501995,MR3591240,MR4047540}. In order to achieve such characterizations for a more general setup, in particular for more general systems of generalized wavelets, we adopt the point of view (and some of the techniques) from \cite{FeFuVo}, with the aim of developing a comprehensive and unified approach. Just as in the mentioned paper, a recurrent theme in our work is the need to understand the influence of structural properties of the underlying group on the problem at hand.
%
In order to extend such characterizations to a more general setup, in particular for expansion to a wide variety of generalized wavelets, we adopt the point of view from \cite{FeFuVo}, with the aim of developing a comprehensive and unified approach. Just as in the mentioned paper, a recurrent theme in our work is the need to understand the influence of structural properties of the underlying group on the problem at hand.

This article contains several substantial generalizations and adaptations of the setup from \cite{FeFuVo}. Most importantly, the investigation of finite degrees of smoothness makes it possible to include larger classes of analyzing wavelets. The results of \cite{FeFuVo} relied on using bandlimited wavelets, i.e. Schwartz functions whose Fourier support is compactly supported, well away from the zero frequency. This leads to wavelet coefficient decay of arbitrary orders, provided the analyzed signal is sufficiently smooth.
Restricting the investigation of smoothness and decay behaviour to finite orders allows to replace the bandlimited analyzing wavelets by a more general class of wavelets, and thus to employ wavelets that are compactly supported in space rather than in frequency. The decisive feature of the wavelet which enables this replacement is the number of vanishing moments (see  Definition \ref{defn:van_mom} for a precise formulation of vanishing moments).
%; for the precise definition of vanishing moments we refer the reader to Definition \ref{defn:van_mom}. 

Using the vanishing moment criteria results in significant consequences in our theory. Firstly, the formulation of direct and inverse theorems for the (near) characterization of $N$-regular directed points in this setting requires lower bounds for the number of vanishing moments on the analyzing wavelet. The recurrent theme of the  dependence of the developed criteria on the underlying group becomes visible here, as well.  
Secondly, wavelets with finitely many vanishing moments can be real-valued. The symmetry properties of these wavelets in the Fourier domain make them essentially incapable of distinguishing the frequency direction $\xi$ from its opposite $-\xi$. In order to address these issues, we replace the wavefront set by the so-called \textit{unsigned wavefront set}, and reformulate the direct and inverse theorems accordingly.  
We emphasize that these issues, in particular the inability of the wavelet transform criteria to distinguish opposite directions of the wavefront set, are also present in the precursor papers \cite{KuLa,Grohs_2011} (albeit not addressed explicitly), but not in \cite{FeFuVo}. We point to subsection \ref{subsect:shearlet_comparison} for an extended comparison of our results to the predecessor papers \cite{KuLa,Grohs_2011}.

\subsection{Overview of the paper and main results}

The chief purpose of this paper is to adapt and extend the previously developed techniques for wavelet characterization of singularities, in order to allow the treatment of finite degrees of smoothness, using compactly supported wavelets. The progress of this endeavour can be traced through  the main results of our paper, which are the following: 
\begin{itemize}
\item Theorem \ref{thm:almost_char}, which establishes criteria for points in the Sobolev wavefront set, using wavelet coefficient decay with respect to bandlimited wavelets; 
\item Theorem \ref{thm:transfer_decay}, showing that these criteria can be transferred between different wavelets, as soon as the cross-kernel associated to the wavelets decays sufficiently quickly; 
\item Theorem \ref{thm:almost_char_Cc} and Corollary \ref{cor:char_vm_infty}, which establish Sobolev wavefront set criteria using compactly supported wavelets;
\item Theorem \ref{thm:almost_char_Cc_sh}, which guarantees the applicability of the previous results for a  large class of shearlet dilation groups. 
\end{itemize}

The remainder of the paper is divided into 5 sections. Section 2 reviews the basic definitions regarding the Fourier and wavelet transforms, $(N-)$regular directed points and similar matters. 
Subsection \ref{subsect:dual_action} recalls the conditions on the dilation group needed to characterize elements of the wavefront set in terms of wavelet coefficient decay,
and recalls related results from \cite{FeFuVo}. This section also contains the definition of the cone-affiliated subsets $K_{o/i}(W,V,R) \subset H$ that are instrumental in the precise formulation of the local wavelet coefficient decay conditions on which the wavelet based criteria for regular directed points are based. This section prepares the necessary background for both the formulation and the proof of Theorem \ref{thm:almost_char}, presented in the subsequent section. This result formulates necessary and sufficient criteria for $N$-regular directed points formulated in terms of wavelet coefficient decay. 
Even though its proof is largely an adaptation of \cite[Theorem 3.5]{FeFuVo}, we include a sketch of the proof of Theorem~\ref{thm:almost_char}, as it will be used in the main result of the subsequent section.

Theorem \ref{thm:almost_char} only allows to use analyzing wavelets that are compactly supported in frequency, which raises the question whether a similar characterization for other types of wavelets is possible as well. 
The extension to more general wavelets requires the introduction of various auxiliary notions, definitions and results. Introducing these is the main purpose of Section \ref{sect:towards_compact_support}. To begin with, since vanishing moment criteria open the door to the use of real-valued wavelets, we introduce the unsigned wavefront set and exhibit its basic properties in Subsection \ref{subsect:unsigned}. We then review the important definitions and results in connection with vanishing moments, which were introduced in  \cite{Fu_coorbit,Fu_atom,FuRe}, in Subsection \ref{subsect:van_moment}. The chief purpose of these results is to guarantee a certain decay of wavelet coefficients, once both the wavelet and the analyzed function are sufficiently smooth, with sufficiently many vanishing moments. We then set out to write an explicit cross-kernel convolution formula (Proposition~\ref{prop:5.3}) which will be used later on.

In Section \ref{sect:transfer}, we use the theory developed in Sections \ref{sect:bandlimited} and \ref{sect:towards_compact_support} to extend our scheme to compactly supported wavelets.
Theorem \ref{thm:transfer_decay} is the central tool for the transfer of wavelet coefficient decay statements between pairs of wavelets with sufficiently many vanishing moments. As such it is of considerable independent interest, since it establishes that the local wavelet coefficient decay of a given signal genuinely reflects properties of the signal, once the analyzing wavelet has a certain oscillatory behaviour. As a surprising byproduct of Theorem~\ref{thm:transfer_decay}, we are able to close a gap in Theorem~\ref{thm:almost_char}; this result is stated in Corollary~\ref{cor:close-gap}.

We subsequently use Theorem~\ref{thm:transfer_decay} to prove Theorem~\ref{thm:almost_char_Cc}, showing that wavelets that are compactly supported in space (rather than in frequency) can also be used for the near characterization of unsigned regular directed points of finite order. As a further application we note in Corollary \ref{cor:char_vm_infty} that real-valued wavelets with vanishing moments of arbitrary order can be used to characterize unsigned regular directed points of infinite order.

The closing section \ref{sect:shearlet} verifies the extensive list of conditions imposed on the dilation group throughout sections \ref{sect:bandlimited} -- \ref{sect:transfer} for the class of \textit{shearlet dilation groups}. We thereby exhibit a large class of dilation groups in arbitrary dimensions to which the main results of this paper are applicable, with concrete estimates for all constants involved in the criteria. These observations are summarized in Theorem \ref{thm:almost_char_Cc_sh}. 
This result also fosters an understanding of the influence of various features of the underlying group on properties of the wavelet transform. 
We close by comparing our results for generalized shearlet groups to the main results of \cite{KuLa,Grohs_2011} in Subsection \ref{subsect:shearlet_comparison}.

\section{Basic definitions regarding regular directed points and wavelet transforms}
\label{sect:intro}

\subsection{$N$-Regular directed points and the wavefront set}

First let us introduce some notation.
Given $R>0$ and $x\in\mathbb{R}^{d}$, we let $B_{R}(x)$ and $\overline{B_{R}}\left(x\right)$
denote the open/closed ball with radius $R$ and center $x$, respectively.
We let $S^{d-1}\subset\mathbb{R}^{d}$ denote the unit sphere. By
a neighborhood of $\xi\in S^{d-1}$, we will always mean a \emph{relatively
open} set $W\subset S^{d-1}$ with $\xi\in W$. Given $R>0$ and an
open set $W\subset S^{d-1}$, we define the (truncated) cone generated by $W$ as
\[
C(W):=\left\{ r\xi'\with\xi'\in W,\, r>0\right\} =\left\{ \xi\in\mathbb{R}^{d}\setminus\left\{ 0\right\} \with\frac{\xi}{\left|\xi\right|}\in W\right\} \,\,\,\text{ and }\,\,\, C(W,R):=C(W)\setminus\overline{B_{R}}(0).
\]
Both sets are clearly open subsets of $\mathbb{R}^{d}\setminus\left\{ 0\right\} $
and thus of $\mathbb{R}^{d}$.

Given a tempered distribution $u$ and an integer $N \in \mathbb{N}$, we call $(x,\xi)\in\mathbb{R}^{d}\times S^{d-1}$
an \textbf{$N$-regular directed point of $u$} if there exists $\varphi\in C_{c}^{\infty}(\mathbb{R}^{d})$,
identically one in a neighborhood of $x$, as well as a $\xi$-neighborhood
$W\subset S^{d-1}$ and 
a constant $C_{N}>0$ with 
\begin{equation}
\forall\,\xi'\in C\left(W\right)~:~\left|\widehat{\varphi u}(\xi')\right|\le C_{N}(1+|\xi'|)^{-N}.\label{eqn:decay_cond_reg_dir}
\end{equation}
Note that this condition effectively only concerns
the behaviour at large frequencies: we may replace
$C\left(W\right)$ in inequality \eqref{eqn:decay_cond_reg_dir} by
$C(W,R)$ for any $R>0$, without changing
the definition of an $N$-regular point. In practice, the suitable $R$ will be chosen depending on $\xi$.
The $N$-wavefront set $WF^N(u)$ consists of all directed points $(x,\xi)$ that are not $N$-regular directed points of $u$.

\subsection{Continuous wavelet transforms in higher dimensions}
We now introduce the necessary notions pertaining to continuous wavelet
transforms in some detail. We fix a closed matrix group $H<{\rm GL}(d,\mathbb{R})$,
the so-called \textbf{dilation group}, and let $G=\mathbb{R}^{d}\rtimes H$.
This is the group of affine mappings generated by $H$ and all translations.
Elements of $G$ are denoted by pairs $(x,h)\in\mathbb{R}^{d}\times H$,
and the product of two group elements is given by $(x,h)(y,g)=(x+hy,hg)$.
The left Haar measure of $G$ is given by ${\rm d}(x,h)=|\det(h)|^{-1}{\rm d}x\,{\rm d}h$, where ${\rm d}x$ and ${\rm d}h$ are the left Haar measures of $\mathbb{R}^d$ and $H$ respectively.
The group $H$ acts on $\mathbb{R}^{d}$ by matrix multiplication. 
The dual action is just the (right) action $\mathbb{R}^{d}\times H\ni(\xi,h)\mapsto h^{T}\xi\in\mathbb{R}^{d}$.

The group $G$ acts unitarily on $\LL^{2}(\mathbb{R}^{d})$ by the \textbf{quasi-regular
representation} defined by 
\begin{equation}
[\pi(x,h)f](y)=|\det(h)|^{-1/2}\cdot f\left(h^{-1}(y-x)\right)~.\label{eqn:def_quasireg}
\end{equation}
The dilation group $H$ is called {\bf irreducibly admissible},
if the associated quasi-regular representation $\pi$ is irreducible and square-integrable. 
By \cite{Fu10}, this property has a characterization in terms of the dual action: $H$ is irreducibly 
admissible if and only if there is a unique conull open orbit $\mathcal{O} = H^T \xi_0$ of the dual action, with the additional
property that the stabilizer groups associated to $\mathcal{O}$ are compact.

Given functions $\psi,f \in \LL^2(\mathbb{R}^d)$, the \textbf{continuous wavelet transform of $f$ with respect to $\psi$} is the function $W_\psi f: G \to \mathbb{C}$, $W_\psi f (x,h) = \langle f, \pi(x,h) \psi \rangle$. In this context, $\psi$ is called the \textbf{analyzing wavelet}.
Moreover, $\pi$ induces an action of $G$ on $\mathcal{S}(\mathbb{R}^{d})$,
the space of Schwartz functions.
We write $\left\langle \cdot\mid\cdot\right\rangle :\mathcal{S}'(\mathbb{R}^{d})\times\mathcal{S}\left(\mathbb{R}^{d}\right)\to\mathbb{C}$
for the natural extension of the $\LL^{2}$-scalar product $\langle\cdot,\cdot\rangle$, which
means $\langle u\mid\psi\rangle:=u\left(\overline{\psi}\right)$.
For future reference, let us observe that a straightforward calculation
yields 
\begin{equation}
\left[\mathcal{F}\left(\pi\left(x,h\right)f\right)\right]\left(\xi\right)=\left|\det\left(h\right)\right|^{1/2}\cdot e^{-2\pi i\left\langle x,\xi\right\rangle }\cdot\widehat{f}\left(h^{T}\xi\right)\label{eq:QuasiRegularOnFourierSide}
\end{equation}
for $f\in\LL^{1}\left(\mathbb{R}^{d}\right)+\LL^{2}\left(\mathbb{R}^{d}\right)$.
Here, as in the remainder of the paper, we use the convention 
\[
\mathcal{F}f\left(\xi\right)=\widehat{f}\left(\xi\right)=\int_{\mathbb{R}^{d}}f\left(x\right)\cdot e^{-2\pi i\left\langle x,\xi\right\rangle }\,{\rm d}x
\]
for the Fourier transform of $f\in L^{1}\left(\mathbb{R}^{d}\right)$.

Using a Schwartz function $\psi$ as analyzing wavelet allows to extend the continuous wavelet transform to the space of tempered distributions in a straightforward manner.
Namely, given a tempered distribution $u\in\mathcal{S}'\left(\mathbb{R}^{d}\right)$
and some $\psi\in\mathcal{S}(\mathbb{R}^{d})$, we define the \textbf{wavelet
transform} of $u$ with respect to $\psi$ by 
\[
W_{\psi}u:G\to\mathbb{C},~(x,h)\mapsto\langle u\mid\pi(x,h)\psi\rangle.
\]
Throughout this paper, wavelets $\psi$ will always come from the Schwartz class, and accordingly the wavelet transform is defined on $\mathcal{S}'(\mathbb{R}^d)$.

A wavelet $\psi \in \mathcal{S}(\mathbb{R}^d)$ is called \textbf{admissible} if it fulfills the \textbf{admissibility
condition} 
\[
C_\psi = \int_{H} |\widehat{\psi}(h^T \xi)|^2 \dd h < \infty\mbox{ for all }\xi\in {\mathcal O}~.
\] 
The most important consequence of the admissibility condition is the \textit{inversion formula} 
\[
 f = \frac{1}{C_\psi} \int_G W_\psi f (x,h) \,  \pi(x,h) \psi \,d(x,h)
\] valid (in the weak operator sense) for all $f \in \LL^2(\mathbb{R}^d)$. Note that this inversion formula does not necessarily apply to tempered distributions. 

\subsection{Technical conditions on the dual action}\label{subsect:dual_action}
We will write $V\Subset\mathcal{O}$ to indicate that $V,\mathcal{O}\subset\mathbb{R}^{d}$ are open sets and that the closure $\overline{V}\subset\mathcal{O}$ is a compact subset of $\mathcal{O}$. 
We assume that the dilation group is irreducibly admissible, and so the associated quasi-regular representation is irreducible and square-integrable. 
As a consequence of this assumption, all the properties required in \cite[Assumption 2.1]{FeFuVo} are guaranteed; we spell this out in the following lemma. We refer to \cite[Lemma 2.2]{FeFuVo} and the text leading up to it for a proof. Observe that part (b) of the following lemma is slightly stronger than the corresponding part (b) from  \cite[Assumption 2.1]{FeFuVo}, but the reasoning from \cite{FeFuVo} in fact shows the stronger statement. In particular, observe that $\xi \in \mathcal{O}$ entails $-\xi \in \mathcal{O}$.
\begin{lem}\label{assume:proper_dual} 
Suppose the dilation group is irreducibly admissible. Then
there exists an open, conull, $H^{T}$-invariant
subset $\mathcal{O}\subset\mathbb{R}^{d}$ with the following properties: 
\begin{enumerate}[label=(\alph*)]
%\item \label{enu:DualActionIsProper}The dual action of $H$ on $\mathcal{O}$ is proper, i.e., for all compact sets $K\subset\mathcal{O}$, the set \[H_{K}:=\left\{ (h,\xi)\in H\times\mathcal{O}\with(h^{T}\xi,\xi)\in K\times K\right\} \] is compact. 
\item \label{enu:DualActionContainsRays}For each $\xi\in\mathcal{O}$,
we have $\mathbb{R}^{\ast}\xi\subset\mathcal{O}$, where $\mathbb{R}^{\ast}:=\mathbb{R} \setminus \{0 \}$. 
\item \label{enu:ExistenceOfAdmissibleFunction}
Any Schwartz function $\psi$ such that $\widehat{\psi}$ is compactly supported
inside $\mathcal{O}$ is admissible, i.e., satisfies
\begin{equation*}
\forall\,\xi\in\mathcal{O}~:~\int_{H}|\widehat{\psi}(h^{T}\xi)|^{2}\,{\rm d}h<\infty.%\label{eq:AdmissibilityCondition}
\end{equation*}
%
%\item \label{enu:PolynomialGrowthOfMeasureOfIntersection}Given $\emptyset\neq V\Subset\mathcal{O}$ and $\xi\in\mathcal{O}$, we define 
%\[H_{\xi,V} := \left\{ h\in H\with h^{T}\xi\in V\right\} =\left(h\mapsto h^{T}\xi\right)^{-1}\left(V\right),\]
%which is a relatively compact open set because of $H_{\xi,V}\subset\pi_{1}\left(H_{\left\{ \xi\right\} \cup\overline{V}}\right)$, where $\pi_{1}$ is the projection on the first coordinate. Moreover, for each $\emptyset\neq V\Subset\mathcal{O}$, there are constants $C,\alpha\geq0$ such that the estimate 
%\[\forall\,\xi\in\mathcal{O}~:~\mu_{H}(H_{\xi,V})\le C\cdot\left(1+\left|\xi\right|\right)^{\alpha}\]is fulfilled, where $\mu_{H}$ denotes the left Haar measure on $H$.
\end{enumerate}
\end{lem}

We next formally define the sets $K_{i}$ and $K_{o}$ which
associate dilations to directions. 
\begin{defn}
Let $\emptyset\neq W\subset S^{d-1}$ be open with $W\subset\mathcal{O}$
(which implies $C\left(W\right)\subset\mathcal{O}$). Furthermore,
let $\emptyset\neq V\Subset\mathcal{O}$ and $R>0$. We define 
\[
K_{i}(W,V,R):=\left\{ h\in H\with h^{-T}V\subset C(W,R)\right\} 
\]
as well as 
\[
K_{o}(W,V,R):=\left\{ h\in H\with h^{-T}V\cap C(W,R)\not=\emptyset\right\} .
\]
If the parameters are provided by the context, we will simply write
$K_{i}$ and $K_{o}$. Here, the subscripts $i/o$ stand for ``inner/outer''.
\end{defn}
These two types of sets are the central tool of our analysis. The
intuition behind their definition is that $K_{i}$ contains all dilations
$h$ with the property that the wavelets $\pi(y,h)\psi$ only ``see''
directions in the cone $C(W,R)$, as long as ${\rm supp}\left(\smash{\widehat{\psi}}\right)\subset V$
holds. 
Here, we used the property ${\rm supp}\left(\mathcal{F}\left[\pi\left(y,h\right)\psi\right]\right)\subset h^{-T}{\rm supp}\left(\smash{\widehat{\psi}}\right)\subset h^{-T}V$
which is immediate from equation \eqref{eq:QuasiRegularOnFourierSide}.
As a result, local regularity in these
directions should entail a decay estimate for the wavelet coefficients
$\left(W_{\psi}u\right)\left(y,h\right)$ with $h\in K_{i}$.

On the other hand, $K_{o}$ contains all those dilations that contribute
to the (formal) \emph{wavelet reconstruction} 
\begin{equation} \label{eqn:waverec_dual}
\widehat{\varphi u}\left(\xi\right)=\int_{\mathbb{R}^{d}}\int_{H}\left(W_{\psi}\varphi u\right)\left(y,h\right)\cdot\left(\mathcal{F}\left[\pi\left(y,h\right)\psi\right]\right)\left(\xi\right)\,\frac{{\rm d}h}{\left|\det\left(h\right)\right|}\,{\rm d}y
\end{equation}
of the frequency content of a (localized) tempered distribution $\varphi u$,
for $\xi\in C(W,R)$, again under the assumption ${\rm supp}\left(\smash{\widehat{\psi}}\right)\subset V$.
Thus, decay estimates for wavelet coefficients with dilations in $h\in K_{o}$
should allow to predict local regularity of $u$ in these directions.

Next, we show how the sets $K_i$ and $K_o$ are related to each other when moving along the points of the orbit.
\begin{lem}\label{lem:cone_inclusion}
Let $\xi_0,\xi_1 \in \mathcal{O}\cap S^{d-1}$, such that $h^T \xi_0 = \xi_1$. 
\begin{enumerate}[label=(\alph*)]
    \item\label{lem:cone_inclusion-1} For every $R_1>0$ and $\xi_1$-neighborhood $W_1$ in $S^{d-1}$, there exists $R_0>0$ and a $\xi_0$-neighborhood $W_0$ such that
    $$h^T C(W_0,R_0) \subseteq C(W_1,R_1).$$
    For all $R_0,R_1>0$ and neighborhoods $W_0,W_1$ satisfying the above inclusion, we have 
    $$K_i(W_1,V,R_1) \supseteq h^{-1} K_i(W_0,V,R_0).$$
    \item\label{lem:cone_inclusion-2} For every $R_0>0$ and $\xi_0$-neighborhood $W_0$ in $S^{d-1}$, there exists $R_1>0$ and a $\xi_1$-neighborhood $W_1$ such that
    $$C(W_1,R_1) \subseteq h^{T} C(W_0,R_0).$$
    For all $R_0,R_1>0$ and neighborhoods $W_0,W_1$ satisfying the above inclusion, we have 
    $$K_o(W_1,V,R_1) \subseteq h^{-1} K_o(W_0,V,R_0).$$
\end{enumerate}
\end{lem}

\begin{proof}
To prove \ref{lem:cone_inclusion-1}, let $W_1, R_1$ be given. Pick $\delta>0$ such that for all $v \in B_\delta(\xi_1)$ one has $\frac{v}{|v|} \in W_1$ (see the proof of \cite[Lemma 4.5]{FeFuVo}). Let $R_0 = R_1 \cdot \| h^{-1} \|$ and $W_0 = (h^{-T} B_\delta (\xi_1)) \cap S^{d-1}$, which is indeed a neighborhood of $\xi_0$
in $S^{d-1}$. 
To prove that $R_0,W_0$ are as desired, let $x \in C(W_0,R_0)$. Hence $x = rw$, with $r>R_0$, $w \in W_0$. Then $v = h^T w \in B_{\delta}(\xi_1)$, and by choice of $\delta$, one has $\frac{v}{|v|} \in W_1$. It follows that 
\[
h^T x = r h^T w = r |v| \cdot \frac{v}{|v|} \in \mathbb{R}^+ \cdot W_1~.
\]
Finally, by choice of $x$,
\[
R_1 = R_0 \| h^{-1} \|^{-1} < |x| \| h^{-1} \|^{-1} =  | h^{-T} h^T x|   \| h^{-1} \|^{-1} 
\le |h^T x| ~,
\] which establishes that $h^{T}x \in C(W_1,R_1)$. This proves the first inclusion in \ref{lem:cone_inclusion-1}. 
Next, assume $h^T C(W_0,R_0) \subseteq C(W_1,R_1)$, and pick $g \in K_i(W_0,V,R_0)$. So we have
\[
g^{-T} V \subseteq C(W_0,R_0) \subseteq h^{-T} C(W_1,R_1), 
\]
and thus $h^{-1} g \in K_i(W_1,V,R_1)$. 

To prove \ref{lem:cone_inclusion-2}, let $W_0, R_0$ be as given in the assumption. We apply \ref{lem:cone_inclusion-1} for $h^{-T} \xi_1 = \xi_0$ to find $R_1,W_1$ such that 
$h^{-T} C(W_1,R_1) \subseteq C(W_0,R_0)$. Multiplying both sides by $h^T$ yields the desired inclusion.
To prove the next inclusion, assume $ C(W_1,R_1) \subseteq h^{T} C(W_0,R_0)$, and 
let $g \in K_o(W_1,V,R_1)$. By definition, this means $g^{-T} V \cap C(W_1,R_1) \not= \emptyset$. As $C(W_1,R_1) \subseteq h^T C(W_0,R_0)$, this implies
\[
g^{-T} V \cap  h^T C(W_0,R_0) \not= \emptyset~,
\]
or $(hg)^{-T} V \cap C(W_0,R_0) \not= \emptyset$, which entails $hg \in K_o(W_0,V,R_0)$.
\end{proof}

We now come to an important technical assumption concerning the dual
action. Given a matrix $h$, we let $\|h\|$ denote the operator norm
of the induced linear map with respect to the euclidean norm. 
\begin{defn}
\label{defn:micro_regular}Let $\xi\in\mathcal{O}\cap S^{d-1}$ and
$\emptyset\neq V\Subset\mathcal{O}$. The dual action is called \textbf{$V$-microlocally
admissible in direction $\xi$} if there exists a $\xi$-neighborhood
$W_{0}\subset \mathcal{O}\cap S^{d-1}$ and some $R_{0}>0$ such that
the following hold: 
\begin{enumerate}[label=(\alph*)]
\item \label{enu:NormOfInverseEstimateOnKo}There exist $\alpha_{1}>0$
and $C>0$ such that 
\[
\|h^{-1}\|\le C\cdot\|h\|^{-\alpha_{1}}
\]
holds for all $h\in K_{o}(W_{0},V,R_{0})$. 
\item \label{enu:NormIntegrability}There exists $\alpha_{2}>0$ such that
\[
\int_{K_{o}(W_{0},V,R_{0})}\|h\|^{\alpha_{2}}\,{\rm d}h<\infty.
\]
\end{enumerate}
The dual action is called \textbf{microlocally admissible in direction
$\xi$} if it is $V$-microlocally admissible in direction $\xi$
for some $\emptyset\neq V\Subset\mathcal{O}$. It is called \textbf{globally
$V$-microlocally admissible} if there is $\emptyset\neq V\Subset\mathcal{O}$
such that the dual action is $V$-microlocally admissible in direction
$\xi$ for all $\xi\in\mathcal{O}\cap S^{d-1}$.
\end{defn}

\begin{remark}\label{rem-local/global-microlocality}
 We note that under the prevailing assumption that $\mathcal{O}$ is a single orbit with compact stabilizers, microlocal admissibility in direction $\xi \in \mathcal{O}$ already entails global microlocal admissibility. Under further additional assumptions such as the cone approximation property, condition \ref{enu:NormOfInverseEstimateOnKo} entails condition \ref{enu:NormIntegrability}; see \cite[Lemma 6.1]{FeFuVo}.
\end{remark}

The main purpose of microlocal admissibility is to establish a systematic relationship between the 
norm of $h \in H$ and the norms of
the frequencies contained in the support of ${\mathcal F}(\pi(y,h)\psi)$. Such norm estimates, which we list below, will be instrumental in our proof. See \cite[Lemma 2.8]{FeFuVo} for a proof of the following.
\begin{lem}
\label{lem:related_to_assumptions}Assume that the closed subgroup $H\leq{\rm GL}\left(\mathbb{R}^{d}\right)$ is irreducibly admissible,
and consequently, the properties listed in Lemma~\ref{assume:proper_dual} hold.
Then $0\notin\mathcal{O}$.
Furthermore, the following hold: 
\begin{enumerate}[label=(\alph*)]
\item \label{enu:XiEstimatedByH}Assume that $\emptyset\neq V\Subset\mathcal{O}$.
Then there exists a constant $C_{1}=C_{1}\left(V\right)>0$ such that,
for all $h\in H$ and all $\xi'\in V$: 
\[
\left|h^{-T}\xi'\right|^{-1}\le C_{1}\cdot\|h\|.
\]

\item \label{enu:HEstimatedByXi}Assume that the dual action fulfills condition
\ref{defn:micro_regular}\ref{enu:NormOfInverseEstimateOnKo}, for
some $\emptyset\not=V\Subset\mathcal{O}$, a suitable $\xi$-neighborhood
$W_{0}\subset S^{d-1}$ and some $R_{0}>0$.
Then there exists $C_{2}>0$ such that
\[
\|h\|\le C_{2}\cdot\left|h^{-T}\xi'\right|^{-\frac{1}{\alpha_1}}
\]
holds for all $h\in K_{o}(W_{0},V,R_{0})$ and $\xi'\in V$, with $\alpha_1$ as in Definition \ref{defn:micro_regular}. 
Consequently, we have: 
\[
{\rm sup}_{h\in K_{o}(W_{0},V,R_{0})}\left\Vert h\right\Vert <\infty.
\]
\item \label{enu:calculations-for-norms} With $\alpha_{1}$ as in Definition \ref{defn:micro_regular}, for all $h\in K_{o}\left(W_0,V,R_0\right)$
\begin{equation}
\left|\det\left(h\right)\right|^{-1}\leq C_{3}\cdot\left\Vert h\right\Vert ^{-d\alpha_{1}}\label{eq:InverseDeterminantEstimate}
\end{equation}
and 
\begin{equation}
1+\left\Vert h^{-1}\right\Vert \leq C_{4}\cdot\left\Vert h\right\Vert ^{-\alpha_{1}},\label{eq:InverseNormEstimate}
\end{equation}
hold, for constants $C_{4}=C_{4}\left(W_{0},V,R_{0}\right)>0$ and 
$C_{3}=C_{3}\left(W_{0},V,R_{0}\right)>0$.
\end{enumerate}
\end{lem}

\section{Criteria for $N$-regular directed points: Bandlimited wavelets}

\label{sect:bandlimited}
Let $\mathbb{N}_{0}=\mathbb{N}\cup\{0\}$. In the following, we use the Schwartz norms
\[
\left|\psi\right|_{N}:=\max_{\alpha\in\mathbb{N}_{0}^{d},\,\left|\alpha\right|\leq N}\sup_{z\in\mathbb{R}^{d}}(1+|z|)^{N}\left|\partial^{\alpha}\psi(z)\right|.
\]
Note that these norms are invariant under complex conjugation. 
Let $u \in \mathcal{S}'(\mathbb{R}^d)$ denote a tempered distribution. We say $u$ is of order $N(u)$ if for all $f \in \mathcal{S}(\mathbb{R}^d)$, we have 
 \[ |\langle u|f \rangle| \le C |f|_{N(u)} ~.\] It is well-known that every tempered distribution $u$ is of finite order.
In the next lemma, we collect various known decay behaviors which will be used in our proofs later on. We refer to \cite[Lemma 3.1]{FeFuVo}.
\begin{lem}
\label{lem:wc_decay_general} Let $\psi,\varphi\in\mathcal{S}(\mathbb{R}^{d})$
and $u\in\mathcal{S}'(\mathbb{R}^{d})$. 
\begin{enumerate}[label=(\alph*)]
\item \label{enu:QuasiRegularSchwartzNormForFunctions}For all $N\in\mathbb{N}_0$
and $(x,h)\in G$, the following inequality holds:
\[
\left|\pi(x,h)\psi\right|_{N}\le C_{N}|\psi|_{N}\cdot\left|\det(h)\right|^{-1/2}(1+\|h^{-1}\|)^{N}\cdot\max\left\{ 1,\|h\|^{N}\right\} \cdot(1+|x|)^{N},
\]
with a constant $C_{N}$ independent of $\psi,x,h$. 
\item \label{enu:WaveletTrafoGeneralEstimateForDistributions}
 Let $N(u)$ denote the order of $u$. Then for all $\left(x,h\right)\in G$,
the following inequality holds: 
\[
|W_{\psi}u(x,h)|\le C\cdot\left|\det(h)\right|^{-1/2}(1+\|h^{-1}\|)^{N(u)}\cdot\max\left\{ 1,\|h\|^{N(u)}\right\} \cdot(1+|x|)^{N(u)},
\]
where $C>0$ is a constant depending on $\psi$, $u$ and $N(u)$ but not on $x,h$. 
\item \label{enu:WaveletTrafoGeneralEstimateForSchwartzFunctions}For all
$N\in\mathbb{N}$ and $(x,h)\in G$, we have 
\[
|W_{\psi}\varphi(x,h)|\le C_{N}|\varphi|_{d+N+1}|\psi|_{N}\cdot\left|\det(h)\right|^{-1/2}(1+\|h^{-1}\|)^{N}\cdot\max\left\{ 1,\|h\|^{N}\right\} \cdot(1+|x|)^{-N}
\]
with $C_{N}$ independent of $\varphi,\psi,h,x$. 
\item \label{enu:WaveletTrafoEstimateForSchwartzFunctionsOnTheCone}Assume
that the dual action is $V$-microlocally admissible at $\xi$ for
some $\emptyset\neq V\Subset\mathcal{O}$ and that $\psi\in\mathcal{S}\left(\mathbb{R}^{d}\right)$
with ${\rm supp}(\widehat{\psi})\subset V$. Choose $R_{0}>0$ and
a $\xi$-neighborhood $W_{0}\subset S^{d-1}$ as in Definition \ref{defn:micro_regular}.
Then 
\[
|W_{\psi}\varphi(x,h)|\le C_{M,N,\psi,W_{0},V,R_{0}}\cdot|\varphi|_{M+N}\cdot\left|\det(h)\right|^{-1/2}\cdot(1+|x|)^{-N}\|h\|^{M}
\]
holds for all $x\in\mathbb{R}^{d}$, $h\in K_{o}\left(W_{0},V,R_{0}\right)$
and $M,N\in\mathbb{N}$, where the constant $C_{M,N,\psi,W_{0},V,R_{0}}$
is independent of $x,h$ and $\varphi$. 
\end{enumerate}
\end{lem}

We next address how wavelet coefficient decay and regular directed
points are affected by certain multiplication operators.
We refer to \cite[Lemma 3.2]{FeFuVo} for the proof of (i) and  \cite[Lemma 3.4]{FeFuVo} for the proof of (ii).
\begin{lem}
\label{lem:reg_points_loclem:wavelet_local}
Let $u\in\mathcal{S}'(\mathbb{R}^{d})$
and $(x,\xi)\in\mathbb{R}^{d}\times S^{d-1}$. Let $\psi\in C_{c}^{\infty}(\mathbb{R}^{d})$
be identically one in some neighborhood of $x$. Then we have:
\begin{itemize}
    \item[(i)] $(x,\xi)$ is an $N$-regular directed point of $u$ iff it is an $N$-regular directed point of $\psi u$.
    \item[(ii)] Let $\varphi\in\mathcal{S}(\mathbb{R}^{d})$ be such that ${\rm supp}(\widehat{\varphi})\subset V$
for some $\emptyset\neq V\Subset\mathcal{O}$. Moreover, assume that the dual action is $V$-microlocally
admissible in direction $\xi$. Then for all $N\in \mathbb {N}$, there exist constants $C_{N}>0$ and $\epsilon>0$, such that the estimate 
\[
\left|W_{\varphi}u(y,h)-\left(W_{\varphi}(\psi u)\right)(y,h)\right|\le C_{N}\|h\|^{N}
\]
holds for all $y\in B_{\epsilon}(x)$ and all $h\in K_{o}(W_{0},V,R_{0})$,
as soon as $R_{0}>0$ and the $\xi$-neighborhood $W_{0}\subset S^{d-1}$
satisfy part \ref{enu:NormOfInverseEstimateOnKo} of the definition
of $V$-microlocal admissibility (Definition \ref{defn:micro_regular}).
\end{itemize} 
\end{lem}

\begin{remark}\label{rem:gap_almost_char}
We can now formulate wavelet criteria for $N$-regular directed points in the case of bandlimited wavelets.
Theorem~\ref{thm:almost_char} is an adaptation of \cite[Theorem 3.5]{FeFuVo} to 
the case of finite smoothness. There are {\em two} noteworthy gaps between
the necessary part (a) and the sufficient part (b): One concerns the set of 
matrices $h \in H$ for which the norm estimate holds. For the necessary condition 
the pertinent set is  $K_{i}(W,V,R)$ (which may well be empty), whereas
the sufficient condition requires a decay estimate for the typically larger set $K_o(W,V,R)$.
This gap has already been observed for the criteria in \cite{FeFuVo}. There exist
criteria, the so-called {\em (weak) cone approximation condition(s)}, on the dual action
that allow to determine when this gap can be closed. These are fulfilled, in particular, for
the shearlet dilation group and its relatives. In Section~\ref{sect:transfer}, we prove Theorem~\ref{thm:transfer_decay}, which allows to derive a decay statement for  sets of the type $K_o(W,V,R)$ from decay assumptions for sets of the type $K_i(W',V,R')$. As a result, we are able to close this gap in Corollary~\ref{cor:close-gap} without assuming the cone approximation conditions. Instead, we required that our analyzing wavelet has vanishing moments of sufficiently  large order. 

The second gap concerns the exponents in the decay estimates for wavelet coefficients:
For $N$-regular directed points $(x,\xi)$, the necessary criterion notes a decay
of order $N-\delta_1$, whereas the sufficient criterion requires order $ \alpha_1 N+\delta_2$,
with $\alpha_1,\delta_1,\delta_2>0$. To our current knowledge, this gap can not be closed (although
possibly reduced); see the rather similar results for shearlets obtained in
\cite{Grohs_2011}. We refer to subsection \ref{subsect:shearlet_comparison} for a more detailed discussion.
\end{remark}

\begin{thm}
\label{thm:almost_char}Assume that the dual action is $V$-microlocally
admissible in direction $\xi$, with parameters $\alpha_1, \alpha_2$ as in the definition. Let $u\in\mathcal{S}'(\mathbb{R}^{d})$
and $(x,\xi)\in\mathbb{R}^{d}\times(S^{d-1}\cap \mathcal{O})$. 
\begin{enumerate}[label=(\alph*)]
\item \label{enu:RegularDirectedPointNecessaryCondition}If $(x,\xi)$
is an $N$-regular directed point of $u$, then there exists a neighborhood
$U$ of $x$, some $R>0$ and a $\xi$-neighborhood $W\subset S^{d-1}\cap \mathcal{O}$
such that for \emph{all} admissible $\psi\in\mathcal{S}(\mathbb{R}^{d})$
with ${\rm supp}(\widehat{\psi})\subset V$, the following estimate
holds: 
\[
\exists\, C >0\,\forall\, y\in U\,\forall\, h\in K_{i}(W,V,R)~:~|W_{\psi}u(y,h)|\le C \cdot\|h\|^{N-\alpha_1 d/2}\!\!.
\]
For each such $\psi$, we even have 
\[
\qquad\qquad \exists\, C>0\,\forall\, y\in U\,\forall\, h\in K_{i}(W,\,\widehat{\psi}^{-1}\left(\mathbb{C}\setminus\left\{ 0\right\} \right),\, R)~:~|W_{\psi}u(y,h)|\le C\cdot\left\Vert h\right\Vert ^{N-\alpha_1 d/2}\!\!.
\]

\item \label{enu:RegularDirectedPointSufficientCondition}Let $\psi\in\mathcal{S}(\mathbb{R}^{d})$
be admissible with ${\rm supp}(\widehat{\psi})\subset V$.  
Assume that there exist a neighborhood $U$ of $x$, a number $R>0$ and a $\xi$-neighborhood $W\subset S^{d-1}\cap \mathcal{O}$ such that
\[\exists\, C>0\,\forall\, y\in U\,\forall\, h\in K_{o}(W,V,R)~:~|W_{\psi}u(y,h)|\le C \cdot\left\Vert h\right\Vert^{\alpha_1 N+\frac{3}{2} \alpha_1 d + \alpha_2}
\]
Then $(x,\xi)$ is an $N$-regular directed point of $u$. 
\end{enumerate}
\end{thm}
\begin{proof}
Part (a) is provided by a verbatim reproduction of the proof of \cite[Theorem 3.5(a)]{FeFuVo}. In a similar way, part (b) is obtained by a straightforward adaptation of the argument for \cite[Theorem 3.5(b)]{FeFuVo}. Since this argument is more involved, we provide a somewhat more detailed sketch. 

For the proof of part \ref{enu:RegularDirectedPointSufficientCondition},
observe that we may assume $u$ to be compactly supported by Lemma
\ref{lem:reg_points_loclem:wavelet_local}. 
Let $\varepsilon_1>0$ be such that $ B_{\varepsilon_{1}}(x)\subseteq U$, and assume that $\varepsilon_1$ satisfies the statement in Lemma~\ref{lem:reg_points_loclem:wavelet_local}(ii).
By assumption,
\begin{equation}
|W_{\psi}u(y,h)|\le C\|h\|^{\alpha_1 (N+\frac{3}{2}d) + \alpha_2}\label{eq:SufficientConditionRepeated}
\end{equation}
holds for all  $y\in B_{\varepsilon_{1}}(x)$ and
$h\in K_{o}(W,V,R)$. With $R_{0}>0$ and $W_{0}\subset S^{d-1}\cap\mathcal{O}$
as in Definition \ref{defn:micro_regular}, we may assume $W\subset W_{0}\subset\mathcal{O}$
and $R>R_{0}$. In particular, this implies $C\left(W,R\right)\subset\mathcal{O}$,
thanks to Lemma~\ref{assume:proper_dual}\ref{enu:DualActionContainsRays}.

Let 
\begin{equation}
\eta:=\int_{\mathbb{R}^{d}}\int_{K_{o}(W,V,R)}W_{\psi}u(y,h)\cdot\pi(y,h)\psi\,\frac{{\rm d}h}{|{\rm det}(h)|}\,{\rm d}y,\label{eqn:def_eta}
\end{equation}
where the integral defining $\eta$ is to be understood in the weak
sense, i.e., given $\varphi\in\mathcal{S}(\mathbb{R}^{d})$, we let
\begin{align}
\langle\eta\mid\varphi\rangle & =\int_{\mathbb{R}^{d}}\int_{\mathbb{R}^{d}}\int_{K_{o}(W,V,R)}W_{\psi}u(y,h)[\pi(y,h)\psi](z)\overline{\varphi(z)}\frac{{\rm d}h}{|{\rm det}(h)|}\,{\rm d}y{\rm d}z\nonumber \\
 & =\int_{\mathbb{R}^{d}}\int_{K_{o}\left(W,V,R\right)}W_{\psi}u\left(y,h\right)\cdot\overline{\left(W_{\psi}\varphi \right)\left(y,h\right)}\,\frac{{\rm d}h}{|{\rm det}(h)|}\,{\rm d}y.\label{eqn:weak_def_eta}
\end{align}
Using parts (b) and (d)  of Lemma \ref{lem:wc_decay_general}  together with Lemma~\ref{lem:related_to_assumptions}, it can be shown that $\eta$ is a well-defined
element of $\mathcal{S}'(\mathbb{R}^{d})$, and the integral in equation
\eqref{eqn:weak_def_eta} converges absolutely for all $\varphi\in\mathcal{S}\left(\mathbb{R}^{d}\right)$. 
Next, we show that $u=\eta$. Indeed, for the compactly supported tempered distribution $u$, its Fourier transform $\widehat{u}$ can be represented as a continuous function on ${\mathbb R}^d$, which has an explicit formula on ${\mathcal O}$ given as follows:
$$\widehat{u}(\xi')=\int_{{\mathbb R}^d}\int_H W_\psi u(y,h){\mathcal F}[\pi(y,h)\psi](\xi') \frac{{\rm d} h}{|\det h|}{\rm d}y,$$
for all $\xi'\in {\mathcal O}$.  Since ${\mathcal O}$ is a co-null set in $\mathbb{R}^{d}$, for every $\varphi\in {\mathcal S}({\mathbb R}^d)$ we have
\begin{eqnarray*}
\langle u\mid \varphi \rangle&=&\langle \widehat{u}\mid \FF^{-1}\varphi \rangle=\int_{{\mathcal O}} \widehat{u}(\xi')\overline{\FF^{-1}\varphi(\xi')}{\rm d}\xi'\\
&=&\int_{{\mathcal O}} \int_{{\mathbb R}^d}\int_H W_\psi u(y,h){\mathcal F}[\pi(y,h)\psi](\xi') \overline{\FF^{-1}\varphi(\xi')}\frac{{\rm d} h}{|\det h|}{\rm d}y {\rm d}\xi'\\
&=&\int_{{\mathcal O}} \int_{{\mathbb R}^d}\int_{K_o(W,V,R)} W_\psi u(y,h){\mathcal F}[\pi(y,h)\psi](\xi') \overline{\FF^{-1}\varphi(\xi')}\frac{{\rm d}h}{|\det h|}{\rm d}y{\rm d}\xi'\\
&=&\int_{{\mathbb R}^d}\int_{K_o(W,V,R)} W_\psi u(y,h)\left(\int_{{\RR^d}} {\mathcal F}[\pi(y,h)\psi](\xi') \overline{\FF^{-1}\varphi(\xi')}{\rm d}\xi'\right)\frac{ {\rm d} h}{|\det h|}{\rm d}y\\
&=&\int_{{\mathbb R}^d}\int_{K_o(W,V,R)} W_\psi u(y,h)\left(\int_{{\RR^d}}[\pi(y,h)\psi](z) \overline{{\varphi(z)}}{\rm d} z\right)\frac{{\rm d} h}{|\det h|}{\rm d}y\\
&=&\int_{{\RR^d}}\int_{{\mathbb R}^d}\int_{K_o(W,V,R)} W_\psi u(y,h)[\pi(y,h)\psi](z) \overline{\varphi(z)}\frac{{\rm d} h}{|\det h|} {\rm d}y{\rm d}z\\
&=&\langle \eta\mid\varphi \rangle.
\end{eqnarray*}

To obtain the desired result, we will show that in a suitable neighborhood
of $x$, the distribution $u$ is given by a $C^{N}$-function. Consider  the auxiliary function $\kappa$ obtained by pointwise evaluation of the
integral defining $\eta$ weakly in equation \eqref{eqn:def_eta}, i.e.
\[
\kappa:B_{\varepsilon_{1}/2}(x)\to\mathbb{C},\, z\mapsto\int_{\mathbb{R}^{d}}\int_{K_{o}(W,V,R)}W_{\psi}u(y,h)[\pi(y,h)\psi](z)\,\frac{{\rm d}h}{\left|\det\left(h\right)\right|}\,{\rm d}y.
\]
Note that  the convergence that is crucial to equate pointwise and weak definitions is guaranteed by the assumptions on wavelet coefficient decay.
We then write $\kappa = \kappa_1 + \kappa_2$, where 
\[
\kappa_{1}(z):=\int_{B_{\varepsilon_{1}}(x)}\int_{K_{o}(W,V,R)}W_{\psi}u(y,h)[\pi(y,h)\psi](z)\,\frac{{\rm d}h}{\left|\det\left(h\right)\right|}\,{\rm d}y.
\]
Note that $\kappa_2$ is obtained by integration over the complement of the ball $B_{\varepsilon_{1}}(x)$, where we make no specific assumptions on the wavelet coefficient decay. 
Using the proof of Theorem~3.5 of \cite{FeFuVo}, we have that $\kappa_2$ is $C^\infty$ inside $B_{\varepsilon_{1}/2}(x)$, and 
$\kappa_1$ is $N$-fold differentiable as soon as the decay order $M$ of the wavelet coefficients fulfills $M\geq  \alpha_1 (N+\frac{3}{2}d) + \alpha_2$. 
Hence $\kappa$ is $C^N$, and for every $\varphi \in \mathcal{S}(\mathbb{R}^d)$ with ${\rm supp}(\varphi) \subset B_{\epsilon_1/2}(x)$, we clearly have 
\[
 \langle \eta | \varphi \rangle = \int_{\mathbb{R}^d} \kappa(z) \overline{\varphi(z)} dz~.
\]
\end{proof}

%%%%%%%%%%%%%%%%%%%%%%%%%%%%%%%%%%%%%%%%%%%%%%%%%%%%%%%%%%%%%%%%%%%%%%%%%%%%%%%%%
%%%%%%%%%%%%%%%%%%%%%%%%%%%%%%%%%%%%%%%%%%%%%%%%%%%%%%%%%%%%%%%%%%%%%%%%%%%%%%%%%
\section{Towards an extension to compactly supported wavelets}
\label{sect:towards_compact_support}

As demonstrated in \cite{Grohs_2011} for shearlets in two dimensions, it is possible to sometimes replace the assumption that $\widehat{\psi}$ is compactly supported by suitable finite vanishing moment conditions, and still obtain near characterizations of $N$-regular directed points.  This allows to use wavelets that are compactly supported {\em in space} rather than in frequency. The chief purpose of this section is to prepare the ground for the desired generalization, by highlighting and clarifying various technical issues that arise in the process.

The idea underlying the extension is reasonably simple: Assume that we have an estimate
\begin{equation} \label{eqn:wdec_psi}
 |W_\psi u(y,h)| \le  C \| h \|^N ~,
\end{equation} valid on a set $B_{\epsilon}(x) \times K_{i/o}(W,V,R)$, with a wavelet $\psi$ satisfying ${\rm supp}(\widehat{\psi}) \subset V \Subset \mathcal{O}$. 
Assume that $\eta \in \mathcal{S}(\mathbb{R}^d)$ is a second admissible wavelet. Under suitable convergence assumptions, it is possible to plug the inversion formula 
\[ \eta = \frac{1}{C_\psi} \int_G \WW_\psi \eta (x,h) \pi(x,h) \psi \dd(x,h)
\]
into the definition of the wavelet transform $\WW_\eta u$, and obtain that the wavelet transforms $\WW_\eta u$ and $\WW_\psi u$ are related by a convolution product: 
\[
 \WW_\eta u = \frac{1}{C_\psi} \WW_\psi u \ast \WW_\eta \psi~.
\] 
Subsection~\ref{subsect:van_moment} will be devoted to formulating the explicit criteria needed for the above identity to hold.
Hence in order to transfer decay estimates from $\WW_\psi u$ to $\WW_\eta u$, we need to understand which conditions on the cross-kernel $\WW_\eta \psi$ allow to transfer the wavelet coefficient decay from $\WW_\psi u$ to $\WW_\eta u$, possibly on a slightly smaller set $K'$. Note also that the same reasoning, with $\WW_\psi \eta$ replacing $\WW_{\eta} \psi$, would also allow to transfer decay estimates for $\WW_\eta u$ to $\WW_\psi u$. This allows us to extend both necessary and sufficient wavelet coefficient decay conditions for $N$-regular directed points.

Our proof strategy also requires the use of methods to estimate and quantify the decay of the cross-kernel $W_\eta \psi$; here we will take advantage of previous work developed in \cite{Fu_coorbit,Fu_atom,FuRe}. However, one further aspect that needs to be addressed in connection with vanishing moment conditions is the fact that these conditions can be fulfilled by {\em real-valued wavelets}. Due to the symmetry condition that real-valuedness of $\psi$ imposes on $\widehat{\psi}$, these wavelets are generally unable to distinguish between directions $\xi,-\xi$, i.e., our wavelet-based near characterizations will describe {\em unsigned} $N$-regular directed points, in the sense of Definition \ref{defn:unsigned_reg_dir} below. We emphasize that this restriction is not due to a general shortcoming of our approach, and that it is also present (rather more implicitly, but not less stringently) in the precursor papers \cite{KuLa,Grohs_2011}. We refer to Subsection \ref{subsect:shearlet_comparison} for a more detailed explanation. But now we start the discussion of unsigned regular directed point.

\subsection{Unsigned regular directed points and real-valued distributions}

\label{subsect:unsigned}

As already mentioned, attempting to derive decay statements from vanishing moment conditions opens the door to the use of real-valued wavelets. 
For these wavelets $\psi$ however, one has $\widehat{\psi}(-\xi) = \widehat{\psi}^\ast (\xi) := \overline{\widehat{\psi}(\xi)}$. Hence by construction, such a wavelet will not be able to distinguish the decay behaviour inside a cone $C$ containing $\xi$ and the cone $-C$, containing $-\xi$. As a consequence, we will restrict attention to real-valued wavelets in the following, and will pay the price that the behaviour in directions $\xi$ and $-\xi$ is indistinguishable. The following definition adapts the notion of $N$-regular directed points to this setting. While we expect that these notions have already been studied in some detail elsewhere, we have not been able to locate a convenient source.

\begin{defn} \label{defn:unsigned_reg_dir}
 For $(x,\xi),(x',\xi')$ we write $(x,\xi) \sim (x',\xi')$ iff $x=x', \xi = \pm \xi'$, and use $\mathbb{R}^d \times S^{d-1}/ \pm$ to denote the set of equivalence classes. 
 
 Given $N \in \mathbb{N}$, $u \in {\mathcal S}'(\mathbb{R}^d)$ and $(x,\pm \xi) \in \mathbb{R}^d \times S^{d-1} / \pm$, we call $(x,\pm \xi)$ an {\bf unsigned $N$-regular directed point of $u$} if both $(x,\xi)$ and $(x,-\xi)$ are $N$-regular directed points of $u$. We call $(x,\pm \xi)$ {\bf unsigned regular directed point} if $(x,\pm \xi)$ is an unsigned $N$-regular directed point for all $N$, or equivalently, if $(x,\xi)$ and $(x,-\xi)$ are regular directed points of $u$.
\end{defn}

It turns out that this notion is closely related to the $N$-regular directed points of the real and imaginary parts of $u$. We first recall the necessary terminology: Given $u \in {\mathcal S}'(\mathbb{R}^d)$, we define $\overline{u} \in {\mathcal S}'(\mathbb{R}^d)$ as $\langle \overline{u}| f \rangle = \overline{ \langle u| \overline{f} \rangle}$; recall that we use $\LL^2$-notation here. 
The \textit{real} and \textit{imaginary parts} of $u$ are then defined in the usual manner as ${\rm Re}(u) = \frac{1}{2}(u + \overline{u})$ and ${\rm Im}(u) = \frac{1}{2i}(u - \overline{u})$. 

It follows that $u = {\rm Re}(u) + i {\rm Im}(u)$, and for real-valued functions $f$, one has 
\[
\langle {\rm Re} (u)| f \rangle = {\rm Re} \langle u| f \rangle~,~  \langle {\rm Im} (u)| f \rangle = {\rm Im} \langle u| f \rangle~.
\] We call $u$ \textit{real-valued} if $u = {\rm Re}(u)$. It is easy to see that a tempered distribution is real-valued iff it maps real-valued Schwartz functions to real values. 

In order to express the just-described operations on tempered distributions in Fourier domain, we need to extend the involution $f^*(x) := \overline{f(-x)}$ from functions to tempered distributions, by letting
\[
\langle u^* | f \rangle = \overline{\langle u| f^* \rangle} ~.
\]

Part (b) of the following lemma makes clear that for real-valued distributions, the distinction between signed and unsigned $N$-regular directed points is moot, and part (c) relates unsigned $N$-regular directed points of $u$ to $N$-regular directed points of ${\rm Re}(u)$ and ${\rm Im}(u)$. 

\begin{lem} \label{lem:unsigned_dir_points}
Let $u \in \mathcal{S}'(\mathbb{R}^d)$.
\begin{enumerate}
    \item[(a)] $\widehat{\overline{u}} = (\widehat{u})^\ast$. 
    \item[(b)] If $u$ is real-valued, then its Fourier transform $\widehat{u}$ fulfills $\widehat{u} = (\widehat{u})^*$. As a consequence, $(x,-\xi)$ is a regular directed point of $u$ iff $(x,\xi)$ is. 
    \item[(c)] $(x, \pm \xi)$ is an unsigned $N$-regular directed point of $u$ iff $(x,\xi)$ is an $N$-regular directed point of ${\rm Re}(u)$ and ${\rm Im}(u)$, or equivalently, iff
 $(x,\pm \xi)$ is an unsigned $N$-regular directed point of ${\rm Re}(u)$ and ${\rm Im}(u)$.
\end{enumerate}
\end{lem}

\begin{proof}
Part (a) is straightforward. Part (b) follows from the fact that the cut-out function $\varphi$ in the definition of $N$-regular directed points can be chosen real-valued, and then (a) applies to the product $\varphi u$, and yields the desired statement. 

For part (c), assume that $(x,\pm \xi)$ is an unsigned $N$-regular directed point of $u$, i.e., we have 
decay of $\widehat{\varphi u}$ of order $N$ inside $C \cup -C$, for a suitable cone $C$ in direction $\xi$. 
Again we choose the cut-out function $\varphi$ to be real-valued, and then $\varphi \overline{u} = \overline{\varphi u}$. 
So ${\rm Re}(\varphi u) = \varphi ({\rm Re}(u))$, and the same holds for the imaginary part as well.
Using this fact, together with part (a), we get
\[
\left(\varphi {\rm Re}(u)\right)^\wedge (\xi') =\left({\rm Re}(\varphi u)\right)^\wedge (\xi')= \frac{1}{2} \left( \widehat{\varphi u} + \widehat{\overline{\varphi u}} \right)(\xi') = \frac{1}{2} \left( \widehat{\varphi u}(\xi') + \overline{\widehat{\varphi u}(-\xi')} \right).
\] 
So, the same decay behaviour for $\left(\varphi {\rm Re}(u) \right)^\wedge$, and similarly for $\left( \varphi {\rm Im}(u)\right)^\wedge$, holds in 
$C\cup -C$.
This establishes that $(x,\xi)$ and $(x,-\xi)$ are $N$-regular directed points of ${\rm Re}(u)$ and ${\rm Im}(u)$; or equivalently, that
$(x,\pm \xi)$ is an unsigned $N$-regular directed point of ${\rm Re}(u)$ and ${\rm Im}(u)$.

The converse direction follows directly from part (b) and $u = {\rm Re}(u) + i {\rm Im}(u)$.
\end{proof}

The next somewhat technical notion formulates a fairly subtle structural assumption on the dilation group $H$ that initially seems a bit opaque. We will repeatedly comment on the role of this assumption to make its impact on our arguments clear. 
\begin{defn}
We call $H_0< H$ a {\bf half-space adapted subgroup} if it is an open subgroup with the property that, for any $\xi_0 \in \mathcal{O}$, one has 
$\mathcal{O} = (H_0^T \xi_0) \cup -(H_0^T \xi_0) $, where the union is {\em disjoint}.
\end{defn}

As a pertinent class of examples with this property we mention the generalized shearlet dilation groups featured in Section \ref{sect:shearlet}. On the other hand, it is clear that not all admissible dilation groups have a half-space adapted subgroup. Note that by definition $H_0$ is necessarily a proper open subgroup of index 2. The standard higher-dimensional wavelet dilation groups $H = \mathbb{R}^+ SO(\mathbb{R}^d)$ are connected, and therefore do not possess a half-space adapted subgroup. 

The chief purpose of introducing the notion of half-space adapted subgroups is that in certain contexts, the integration domain $H$ may be swapped for the smaller subset $H_0$, and this will be instrumental in guaranteeing the necessary decay condition. The following lemma spells this phenomenon out for the inversion formula associated to real-valued wavelets. 

\begin{lem} \label{lem:integrate_halfspace}
Let $H_0< H$ denote a half-space adapted subgroup, and let $\psi$ denote a real-valued admissible wavelet for $H$. Then, for all $f \in \LL^2(\mathbb{R}^d)$, one has 
\[
f = \frac{2}{C_\psi} \int_{G_0} \WW_\psi f (x,h) \pi(x,h) \psi \dd(x,h) ~,  
\] where $G_0 = \mathbb{R}^d \rtimes H_0$.
\end{lem}
\begin{proof}
Writing $\mathcal{O}^+ = H_0^T \xi_0$, one has $\mathcal{O} = \mathcal{O}^+ \cup - \mathcal{O}^+$. The admissibility assumption on $\psi$ entails that
\[
C_\psi = \int_H |\widehat{\psi}(h^T \xi_0)|^2 \dd h < \infty~, 
\] and accordingly, that $\| \WW_\psi f \|_{\LL^2(G)}^2 = C_\psi \| f \|^2$, as well as 
\[
f = \frac{1}{C_\psi} \int_G \WW_\psi f(x,h) \pi(x,h) \psi \dd(x,h)~. 
\]
By assumptions on $\psi$ and $H_0$, we have
\[
\int_{H_0} |\widehat{\psi}(h^T \xi_0)|^2\dd h = \int_{H_0} |\widehat{\psi}(-h^T \xi_0)|^2 \dd h = \frac{C_\psi}{2}~. 
\] Since $\mathcal{O}^+ \cup (-\mathcal{O}^+)= \mathcal{O}$ is conull, the admissibility condition for non-irreducible quasi-regular representations 
(see e.g.~Theorem 1 of \cite{FuMa}) allows to conclude for the restriction of $\WW_\psi f$ to $G_0$ that $\| \WW_\psi f\|_{\LL^2(G_0)}^2 = \frac{C_\psi}{2}
\| f \|^2$, and accordingly 
\[
f = \frac{2}{C_\psi} \int_{G_0} \WW_\psi f(x,h) \pi(x,h) \psi \dd(x,h)~.
\]
\end{proof}

\subsection{Vanishing moments and wavelet coefficient decay} 
\label{subsect:van_moment}

Our methods and techniques build on previous work on wavelet coefficient decay estimates from vanishing moment conditions. In this subsection, we shall roughly sketch the pertinent results from \cite{Fu_coorbit,Fu_atom,FuRe}. 

The first question that arises concerns the precise nature of the vanishing moment criteria. In higher dimensions, these depend on the dilation group, or more precisely, on the open dual orbit $\mathcal{O}$:
\begin{defn} \label{defn:van_mom}
 Let $r \in \mathbb{N}$ be given.
We say $\psi \in \mathcal{S}(\mathbb{R}^d)$ {\bf has vanishing moments in $\mathcal{O}^c$ of order $r$}  if all
derivatives $\partial^\alpha \widehat{\psi}$ of degree $|\alpha|< r$ are vanishing on $\mathcal{O}^c$.
\end{defn}
Clearly, any bandlimited wavelet, i.e., any Schwartz function $\psi$ with $\widehat{\psi} \in C_c^\infty (\mathcal{O})$ has vanishing moments of all orders. 
%\marginpar{\tiny{MG: Removed  under suitable conditions...}}
Note that the vanishing moment conditions are equivalent to 
\[
 \forall |\alpha| < r,\forall \xi \in \mathcal{O}^c ~:~\int_{\mathbb{R}^d} x^\alpha \psi(x) e^{-2 \pi i \langle \xi, x \rangle} \dd x = 0~. 
\]
We next define an auxiliary function $A: \mathcal{O} \to \mathbb{R}^+$ that is used in the decay estimates.  
Given any point $\xi \in \mathcal{O}$, let ${\rm
dist}(\xi,\mathcal{O}^c)$ denote the  distance of $\xi$ to the set
$\mathcal{O}^c$, and define
\begin{equation} \label{eqn:A_euclid}
 A(\xi) := \min \left( \frac{{\rm dist}(\xi,\mathcal{O}^c)}{1+\sqrt{|\xi|^2-{\rm dist}(\xi,\mathcal{O}^c)^2}},
 \frac{1}{1+|\xi|} \right)~. 
\end{equation}
By definition, $A$ is a continuous function with $A(\cdot) \le 1$.
If $\eta \in \mathcal{O}^c$ denotes an element of minimal distance to $\xi$, the fact that $\mathbb{R}^+ \cdot \eta  \subset \mathcal{O}^c$ then entails that $\eta$ and $\xi-\eta$ are orthogonal with respect to the standard scalar product on $\mathbb{R}^d$, and we obtain the more transparent expression
\begin{equation} \label{eqn:A_noneuclid}
  A(\xi) = \min_{\eta\in\mathcal{O}^c} \left( \frac{|\xi-\eta|}{1+|\eta|},
 \frac{1}{1+|\xi|} \right)~. 
\end{equation}
We define a further family of auxiliary functions $\Phi_\ell : H \to \mathbb{R}^+ \cup \{ \infty \}$, for $\ell \in \mathbb{N}$, via 
\begin{equation} \label{eqn:def_Phi_ell}
 \Phi_\ell(h) =  \int_{\mathbb{R}^d} A(\xi)^\ell A(h^T \xi)^\ell \dd\xi~.
\end{equation}
The motivation for the definitions of $A$ and $\Phi_\ell$ is best explained as follows:
Two functions $\eta, \psi \in \mathcal{S}(\mathbb{R})$ are admissible wavelets if and only if $\widehat{\eta} (0) = \widehat{\psi} (0) = 0$. The usual arguments relating smoothness of a function and the decay of its Fourier transform (and vice versa), together with the convolution theorem, allow us to estimate
\begin{equation}\label{eq:motiv}
 | \WW_\psi \eta (x,h)| = |\eta \ast (\pi(0,h) \psi^*) (x)| \le C (1+|x|)^{-N} |h|^{1/2}  \sum_{j=0}^{N} \left\| \partial^j\left(\widehat{\eta} \overline{\widehat{\psi}(h \cdot) }\right) \right\|_1~.
\end{equation}
Here the sum 
\[
 \sum_{j=0}^{N} \left\| \partial^j\left(\widehat{\eta} \overline{\widehat{\psi}(h \cdot)} \right) \right\|_1  
\] can be understood as a term measuring the overlap between $\widehat{\eta}$ and the dilate of $\widehat{\psi}$, and it goes to zero as $h \to 0, h \to \infty$, with a speed depending on the number of vanishing moments of $\eta$ and $\psi$. The analog of this overlap term in the higher-dimensional setting is essentially provided by the auxiliary function $\Phi_\ell$. 

The following result guarantees that the set of compactly supported Schwartz functions with vanishing moments of a given order is nonempty. It follows from Lemma 4.1 from \cite{Fu_coorbit}.

\begin{lem} 
For any $r>0$ there exists a nonzero compactly supported $\psi \in \mathcal{S}(\mathbb{R}^d)$ having vanishing moments in $\mathcal{O}^c$ of order $r$. 
\end{lem}

We next quote the following general estimate \cite[Lemma 3.7]{Fu_atom}, which can be viewed as a generalization of \eqref{eq:motiv}:
\begin{lem} \label{lem:decay_est_wc}
 Let $0 < m < r$, and let $\eta, \psi \in \mathcal{S}(\mathbb{R}^d)$ denote functions with vanishing moments of order $r$ in $\mathcal{O}^c$. Then 
 there exists $C>0$ independent of $x,h,\eta,\psi$ such that
 \[
 \left|  \WW_\psi \eta  (x,h) \right| \leq C  |\widehat{\eta}|_{r} |\widehat{\psi}|_{r} (1+|x|)^{-m}|\det(h)|^{1/2} (1+\| h \|)^m \Phi_{r-m}(h)~. 
 \]
\end{lem}
Next, we use the above lemma to obtain estimates on $W_\psi\eta$.  Recall that ${\mathcal O}$ is an open orbit in $\widehat{\RR^d}$, with the property that the stabilizer groups associated with the $H$-action on  ${\mathcal O}$ are compact. This property allows us to transfer the Lebesgue measure on $\widehat{\RR^d}$ to the Haar measure on $H$, i.e. for fixed $\xi_0\in\OO$ and every $F\in C_c({\mathcal O})$ we have
\[\int_{\mathcal O} F(\xi) \dd \xi=\int_H F(h^T\xi_0)\frac{|\det(h)|}{\Delta_H(h)}\dd h.\]
Fix an element $\xi_0\in\OO$, and define $A_H:H\rightarrow [0,1]$ to be $A_H(h)=A(h^T\xi_0)$.
Our computations in the following proposition are inspired by techniques developed in \cite{FuRe}. 
 
 \begin{prop}\label{prop:decay-cross-kernel}
 Assume that there exists a  positive number $\ell_1\in\RR$ and a constant $C$ so that for every $h\in H$,
 \begin{eqnarray}\label{assump}
 \|h^{\pm1}\|A_H(h)^{\ell_1}&\leq& C.
 \end{eqnarray}
Let $\beta_1, \beta_2,\beta_3$ be positive numbers with $\beta_1\in \NN$, and consider functions $\psi_1,\psi_2\in {\mathcal S}(\mathbb{R}^d)$ with vanishing moments of order $r$  in $\mathcal{O}^c$.
If $r>d/2+\ell_1({\beta_1+\beta_2}+\beta_3)+\beta_1$ then for every $(x,h)\in G$, 
 \[
 \left|  \WW_{\psi_1} \psi_2  (x,h) \right| \leq D(1+|x|)^{-\beta_1} (1+\|h\|)^{-\beta_2}(1+\|h^{-1}\|)^{-\beta_3},
 \]
 where the constant $D$ depends on $\psi_1,\psi_2$,  $\beta_1,\beta_2,\beta_3$, and $r,\ell_1$, and does not depend on $x,h$.
 \end{prop}
 \begin{proof}
Without loss of generality, we can assume that $C\geq 1$.
 By Lemma \ref{lem:decay_est_wc}, we have 
 \[
 \left|  \WW_{\psi_1} \psi_2  (x,h) \right| \leq C_1(1+|x|)^{-\beta_1} |\det(h)|^{1/2} (1+\| h \|)^{\beta_1} \Phi_{r-\beta_1}(h).
 \]
 Hence it remains to show that the map $h\mapsto  |\det(h)|^{1/2} (1+\| h \|)^{\beta_1+\beta_2} (1+\| h^{-1} \|)^{\beta_3}\Phi_{r-\beta_1}(h)$ is bounded.
To do so, define $m(h)= |\det(h)|^{1/2} (1+\| h \|)^{\beta_1+\beta_2} (1+\| h^{-1} \|)^{\beta_3}$, and observe that $m$ is submultiplicative, as $\beta_1,\beta_2, \beta_3>0$. 
Our goal is then to prove that $m\Phi_{r-\beta_1}$ belongs to $L^\infty(H)$.

For $\ell\in\NN$, using the submultiplicativity of $m$, we have
\begin{eqnarray*}
m(h)\Phi_\ell(h)&=&m(h)\int_{\mathbb{R}^d} A(\xi)^\ell A(h^T \xi)^\ell \dd\xi\\
&=&m(h)\int_{H} A(k^T\xi_0)^\ell A(h^Tk^T \xi_0)^\ell \frac{|\det(k)|}{\Delta_H(k)} \dd k\\
&=&m(h)\int_{H} A_H(k)^\ell A_H(kh)^\ell \frac{|\det(k)|}{\Delta_H(k)} \dd k\\
&=&m(h)\int_{H} |\det(k)|^{-1} {A_H}(k^{-1})^\ell A_H(k^{-1}h)^\ell \dd k,\\
&\leq&\int_{H} \left(|\det(k)|^{-1} m(k){A_H}(k^{-1})^\ell\right)\left(m(k^{-1}h)A_H(k^{-1}h)^\ell\right) \dd k\\
&=&(B_1*B_2)(h),
\end{eqnarray*}
where we use the auxiliary functions 
$B_1, B_2:H\rightarrow [0,\infty) $ defined as $B_1(k)=|\det(k)|^{-1} m(k){A_H}(k^{-1})^\ell$ and $B_2(k)=m(k)A_H(k)^\ell$. Now, it is only left to show that if $\ell>\ell_1({\beta_1+\beta_2}+\beta_3)+\frac{d}{2}$ then $B_1, \widetilde{B_2}\in \LL^2(H)$, where we use the  notation $\widetilde{B_2}(k)=B_2(k^{-1})$ whenever $k\in H$. Indeed, if that is the case, we will have $\|m\Phi_\ell\|_\infty\leq \|B_1\|_2\|\widetilde{B_2}\|_2<\infty$. 

Note that $(1+\|h^{\pm1}\|)A_H(h)^{\ell_1}= A_H(h)^{\ell_1}+\|h^{\pm1}\|A_H(h)^{\ell_1}\leq 2C $. If $\ell>d/2+\ell_1({\beta_1+\beta_2}+\beta_3)$, we have
\begin{eqnarray*}
\|B_1\|_2^2&=&\int_H |\det(k)|^{-2} m(k)^2{A_H}(k^{-1})^{2\ell} \dd k\\
&=&\int_H |\det(k)|^{-2}  |\det(k)| (1+\| k \|)^{2\beta_1+2\beta_2} (1+\| k^{-1} \|)^{2\beta_3}{A_H}(k^{-1})^{2\ell} \dd k\\
&\leq&(2C)^{2(\beta_1+\beta_2+\beta_3)}\int_H |\det(k)|^{-1} {A_H}(k^{-1})^{2\ell-2\ell_1({\beta_1+\beta_2}+\beta_3) } \dd k\\
&=&(2C)^{2(\beta_1+\beta_2+\beta_3)}\int_H  {A_H}(k)^{2\ell-2\ell_1({\beta_1+\beta_2}+\beta_3) }\frac{|\det(k)|}{\Delta_H(k)} \dd k\\
&=&(2C)^{2(\beta_1+\beta_2+\beta_3)}\int_{\OO} {A}(\xi)^{2\ell-2\ell_1({\beta_1+\beta_2}+\beta_3) }\dd \xi\\
&\leq&(2C)^{2(\beta_1+\beta_2+\beta_3)}\int_{\OO} \left(\frac{1}{1+|\xi|}\right)^{2\ell-2\ell_1({\beta_1+\beta_2}+\beta_3) }\dd \xi,\\
\end{eqnarray*}
which is a bounded integral provided that $2\ell-2\ell_1({\beta_1+\beta_2}+\beta_3)>d$.

Also,
\begin{eqnarray*}
\|\widetilde{B_2}\|_2^2&=&\int_H m(k^{-1})^2A_H(k^{-1})^{2\ell} \dd k\\
&=&\int_H |\det(k)| (1+\| k \|)^{2(\beta_1+\beta_2)} (1+\| k^{-1} \|)^{2\beta_3}A_H(k)^{2\ell}\frac{1}{\Delta_H(k)} \dd k\\
&\leq&(2C)^{2(\beta_1+\beta_2+\beta_3)}\int_H{A_H}(k)^{2\ell-2\ell_1({\beta_1+\beta_2}+\beta_3) }\frac{ |\det(k)|}{\Delta_H(k)} \dd k\\
&=&(2C)^{2(\beta_1+\beta_2+\beta_3)}\int_{\OO} {A}(\xi)^{2\ell-2\ell_1({\beta_1+\beta_2}+\beta_3) }\dd \xi\\
&\leq&(2C)^{2(\beta_1+\beta_2+\beta_3)}\int_{\OO} \left(\frac{1}{1+|\xi|}\right)^{2\ell-2\ell_1({\beta_1+\beta_2}+\beta_3) }\dd \xi,\\
\end{eqnarray*}
which is a bounded integral provided that $2\ell-2\ell_1({\beta_1+\beta_2}+\beta_3)>d$. Finally, letting $D=C_1\|m\Phi_{r-\beta_1}\|_\infty$ finishes the proof. 
 \end{proof}

The estimates in the previous proposition motivate the introduction of a further family of auxiliary functions. 
 \begin{lem}\label{lem:theta-decays}
 For $s,t\in \RR^+$ and $p\in \RR$, let $\Theta_{s,t}:H\rightarrow (0,\infty)$ be defined as 
 $$\Theta_{s,t}(h)=(1+\|h\|)^{-s}(1+\|h^{-1}\|)^{-t}.$$
If $s\geq \dim H+d+1$ and $t\geq \dim H +d+dp$ then $|\det(\cdot)|^{-p}\Theta_{s,t}\in \LL^1(H)$.
 \end{lem}
 \begin{proof}
We first recall the inequalities $|\det(h)|\leq (1+\|h\|)^d$ and $\Delta_H(h)\leq ((1+\|h\|) (1+\|h^{-1}\|))^{\dim H}$. The first one is known as Hadamard's inequality; for the second one we refer to the proof of Lemma 2.8 of \cite{FuRe}.

These inequalities allow to obtain the following useful estimate:
\begin{eqnarray}
|\det(h)|^{-p}\Theta_{s,t}(h)\frac{\Delta_H(h)}{|\det(h)|}&=&|\det(h^{-1})|^{p+1}\Theta_{s,t}(h)\Delta_H(h)\nonumber\\
&\leq&(1+\|h^{-1}\|)^{d(p+1)}(1+\|h\|)^{-s}(1+\|h^{-1}\|)^{-t}{((1+\|h\|) (1+\|h^{-1}\|))^{\dim H}}\nonumber\\
&=&\Theta_{s-\dim H,t-\dim H-d-dp}(h)~.\label{eq:new1}
\end{eqnarray}
Fix an arbitrary $\xi_0\in \OO$, and define $\widetilde{\Theta}_{s,t}: H\rightarrow (0,\infty)$ as $\widetilde{\Theta}_{s,t}(h)=\int_{H_{\xi_0}}\Theta_{s,t}(kh)\dd_{H_{\xi_0}}k$, where $H_{\xi_0}$ is the stabilizer subgroup at $\xi_0$ with Haar measure $\dd_{H_{\xi_0}}k$. Note that the integral defining $\widetilde{\Theta}_{s,t}$ is bounded, as $H_{\xi_0}$ is compact. Next, we observe that $\widetilde{\Theta}_{s,t}$ can be bounded from above and below by multiples of  ${\Theta}_{s,t}$, i.e., there exist positive constants $a_{s,t}, b_{s,t}$ such that 
$$a_{s,t}\Theta_{s,t}\leq \widetilde{\Theta}_{s,t}\leq b_{s,t}\Theta_{s,t}.$$
Indeed, the inequality $\| g h \| \le \| g \| ~\|h \|$ entails $\Theta_{s,t}(gh) \ge \Theta_{s,t}(g) \Theta_{s,t}(h)$. Hence we get
 \begin{eqnarray*}
\widetilde{\Theta}_{s,t}(h)=\int_{H_{\xi_0}}\Theta_{s,t}(kh)\dd_{H_{\xi_0}}k\geq\Theta_{s,t}(h)\int_{H_{\xi_0}}\Theta_{s,t}(k)\dd_{H_{\xi_0}}k=a_{s,t}\Theta_{s,t}(h),
 \end{eqnarray*}
where $a_{s,t}:=\int_{H_{\xi_0}}\Theta_{s,t}(k)\dd_{H_{\xi_0}}k$. On the other hand, the inequality $\|h\|\leq \|kh\|\|k^{-1}\|$ implies that $(1+\|kh\|)^{-1}\leq (1+\|h\|)^{-1}(1+\|k^{-1}\|)$, which in turn implies that 
\begin{eqnarray*}
\int_{H_{\xi_0}}\Theta_{s,t}(kh)\dd_{H_{\xi_0}}k\leq \Theta_{s,t}(h)\int_{H_{\xi_0}}(1+\|k^{-1}\|)^s(1+\|k\|)^t\dd_{H_{\xi_0}}k=b_{s,t}\Theta_{s,t}(h),
\end{eqnarray*}
where $b_{s,t}:=\int_{H_{\xi_0}}(1+\|k^{-1}\|)^s(1+\|k\|)^t\dd_{H_{\xi_0}}k$. Note that both $a_{s,t},b_{s,t}$ are finite, as  integrals of continuous functions on a compact group, and that $a_{s,t}>0$.

The function $\widetilde{\Theta}_{s,t}$ can be regarded as a well-defined function on $\OO$, which we denote by $F_{s,t}$, in the following manner:
 $$F_{s,t}:\OO\rightarrow (0,\infty), \ F_{s,t}(h^T\xi_0)=\widetilde{\Theta}_{s,t}(h),$$
since $\widetilde{\Theta}_{s,t}(lh)=\widetilde{\Theta}_{s,t}(h)$ for every $h\in H$ and $l\in H_{\xi_0}$. 
 We will use $F_{s,t}$ to relate the 1-norm of $\Theta_{s,t}$ to an integral on the orbit $\OO$. First observe that for $\xi=h^{T}\xi_0$, we clearly have $\frac{|\xi|}{|\xi_0|}\leq \|h\|$, which implies that
 $(1+\|h\|)^{-1}\leq (1+\frac{|\xi|}{|\xi_0|})^{-1}$. 
Hence, as $t,s\geq 0$ we have
 $$F_{s,t}(\xi)=F_{s,t}(h^T\xi_0)=\widetilde{\Theta}_{s,t}(h)\leq b_{s,t}(1+\|h\|)^{-s}(1+\|h^{-1}\|)^{-t}\leq b_{s,t}\left(1+\frac{|\xi|}{|\xi_0|}\right)^{-s}~.$$
 
Letting $s'=s-\dim H$ and $t'=t-d-dp-\dim H$, from the inequality \eqref{eq:new1} we get
\begin{eqnarray*}
\int_{H}|\det(h)|^{-p}\Theta_{s,t}(h)\dd h&\leq&\int_{H}\Theta_{s',t'}(h)\frac{|\det(h)|}{\Delta_H(h)}\dd h\\
&\leq&\frac{1}{a_{s',t'}}\int_{\OO}F_{s',t'}(\xi)\dd \xi\\
&\leq&\frac{b_{s',t'}}{a_{s',t'}}\int_{\OO}\left(1+\frac{|\xi|}{|\xi_0|}\right)^{-s'}\dd\xi,
\end{eqnarray*}
which is a bounded integral if $s'\geq d+1$ and $t'\geq 0$.
 \end{proof}
 %%%%% %%%%% %%%%% %%%%% %%%%% %%%%%%%%%%%%%%%%%%%%%%%%%%%%%%%%%%%%%%%%%%%%%
 As a first application of these estimates, we can formulate explicit criteria for the formula $W_\eta u = W_\psi u \ast W_\eta \psi$ to hold.  
 \begin{prop}\label{prop:5.3}
  Assume that there exist positive numbers $\ell_1,C\in\RR^+$ so that for every $h\in H$,
 \begin{eqnarray}\label{assump_1}
 \|h^{\pm1}\|A_H(h)^{\ell_1}&\leq& C~.
 \end{eqnarray}
 Let $u \in \mathcal{S}'(\mathbb{R}^d)$ denote a tempered distribution of order $N(u)$.
 Let $\eta, \psi \in \mathcal{S}(\mathbb{R}^d)$ denote functions with vanishing moments of order $r$ in $\mathcal{O}^c$, where 
 \[
 r >  2d + N(u) + \ell_1(5 + 2 \dim(H) + 6d + 3N(u))+2~.
 \] 
 Then
 \[
 W_\eta u =  \frac{1}{C_\psi} W_\psi u \ast W_\eta \psi
 \] holds with absolute convergence of the convolution integral. If $H_0< H$ is a half-space adapted subgroup and $\psi$ is real-valued, convolution over $G$ can be replaced by convolution over $G_0 = \mathbb{R}^d \rtimes H_0$, with twice the normalization factor. 
 \end{prop}
 \begin{proof}
 We first recall the wavelet inversion formula 
 \begin{equation} \label{eqn:bochner}
 \eta = \frac{1}{C_\psi} \int_G W_\psi \eta(x,h) \pi(x,h) \psi \dd(x,h)~,
 \end{equation}
 which is valid in the weak $\LL^2$-sense. The proof is based on an alternative interpretation of the right hand side as a Bochner integral converging in $| \cdot |_M$, for a suitable choice of $M$. More precisely, we let 
 \[
 M = \left\{ \begin{array}{cc} N(u) + \frac{d}{2} +1 & d \mbox { even } \\ N(u) + \frac{d+1}{2} & d \mbox{ odd } \end{array} \right. ~.   
 \]

 To see that the desired convergence holds, we introduce the quantities
 \begin{equation} \label{eqn:beta_cond_conv}
 \beta_1 = M+d+ 1~,~\beta_2 = 2+ \dim(H)+\frac{3}{2}d+N(u)~,~\beta_3 = 1 + \dim(H) + 3d  + N(u)~, 
 \end{equation}
 and observe that 
 \begin{equation*} 
 r-d/2 - \ell_1(\beta_1+\beta_2+\beta_3) - \beta_1 >0~.
 \end{equation*}
 This implies, for sufficiently small $\epsilon>0$ that 
\begin{equation}\label{eqn:beta_cond_conv2}
  r-d/2 - \ell_1(\beta_1 + \epsilon +\beta_2+\beta_3) - \beta_1 -\epsilon >0~. 
\end{equation}
 We can then estimate the norm of the integrand in (\ref{eqn:bochner}) as follows: 
 \begin{eqnarray*}
 \lefteqn{\left| (W_\psi \eta)(x,h) \pi(x,h) \psi \right|_{M}} \\
  & \le & C_{M}|\psi|_{{M}}|W_\psi \eta (x,h)| (1+|x|)^{{M}} {|\det(h)|^{-1/2}}(1+\| h \|)^{{M}} (1+\|h^{-1} \|)^{{M}} \\
  & \le & {C'} (1+|x|)^{{{M}}-\beta_1-\epsilon}{|\det(h)|^{-1/2}} (1+\|h\|)^{{{M}}-\beta_2} (1+\| h^{-1} \|)^{{{M}}-\beta_3} \\
  & = & {C'} (1+|x|)^{{{M}}-\beta_1-\epsilon} |\det(h)|^{-1/2} \Theta_{\beta_2-{{M}},\beta_3-{{M}}}(h)~,
 \end{eqnarray*}
 where the constant $C'$ does not depend on $x,h$.
 Here the first inequality is due to Lemma \ref{lem:wc_decay_general}(a), and the second one follows from Proposition \ref{prop:decay-cross-kernel}, which is applicable thanks to \eqref{eqn:beta_cond_conv2}.
 Recall that $\dd(x,h)=\frac{\dd x \dd h}{|\det(h)|}$.
 Now Lemma \ref{lem:theta-decays} and our choice of $\beta_2,\beta_3$ imply (via $M \le N(u)+1+\frac{d}{2}$) that $ |\det(h)|^{{-3/2}} \Theta_{\beta_2-M,\beta_3-M}$ is integrable over $H$. Also, since $\beta_1+\epsilon-M>d$, the function $(1+|x|)^{M-\beta_1-\epsilon}$ is integrable over ${\mathbb R}^d$, and thus the right-hand side of (\ref{eqn:bochner}) is a Bochner integral converging in the completion $X$ of $\mathcal{S}(\mathbb{R}^d)$ with respect to $|\cdot|_M$. In fact, the inequality $d < 2M$ entails $\| f \|_2 \le |f|_M$, which implies that $X$ embeds continuously into $\LL^2(\mathbb{R}^d)$. Thus the limit of the Bochner integral coincides with the weak $\LL^2$-limit, i.e., with $\eta$. So the equation given in (\ref{eqn:bochner}) holds in $X$ as well. Since $N(u) \le M$, $u$ extends to a continuous functional on $X$, which entails that we can interchange evaluation of $u$ and Bochner integration, giving for all $(x',h')$
\begin{eqnarray*}
W_\eta u(x',h') & = & \langle \pi(x',h')^{-1} u, \eta \rangle \\
& = & \frac{1}{C_\psi}\int_G \overline{W_\psi \eta (x,h)} \langle \pi(x',h')^{-1} {u}, \pi(x,h) \psi \rangle \dd(x,h) \\ 
& = & \frac{1}{C_\psi}\int_G \langle u, \pi((x',h')(x,h)) \psi \rangle W_\eta \psi ((x,h)^{-1}) \dd(x,h) \\
& = & \frac{1}{C_\psi} \int_G W_\psi u (x,h) W_\eta \psi ((x,h)^{-1}(x',h')) \dd(x,h) \\
& = & \frac{1}{C_\psi}(W_\psi u \ast W_\eta \psi)(x',h')~. 
\end{eqnarray*}
Here absolute convergence of the convolution integral is guaranteed by Bochner convergence. 

The statement for real-valued wavelets {$\psi$} then follows from Lemma \ref{lem:integrate_halfspace}.
 \end{proof}

\section{Criteria for $N$-regular directed points: Compactly supported wavelets}
\label{sect:transfer}
%%%%%%%%%%%%%%%%%%%%%%%%%%%%%%%%%%%%%%%%%%%%%%%%%%%%%%%%%%%%%%%%%%%%%%%%%%%%%%%%%%%%%%
After the groundwork laid in the previous section, we can now establish the first major result of our paper, namely Theorem \ref{thm:transfer_decay}. It shows that the decay order of wavelet coefficients is stable under exchange of wavelets, as long as both wavelets in the exchange have a certain minimal number of vanishing moments. As an auxiliary notion to guarantee this transfer, one further somewhat technical condition needs to be formulated. 
The intuition behind this definition and its necessity is discussed in Remark~\ref{rem:overlap_control}.

\begin{defn} \label{defn:overlap_control}
Let $V\Subset\mathcal{O}$.
We say that $H$ \emph{fulfills the $V$-overlap control condition with parameters $\gamma_0 \ge 0, \gamma_1,\gamma_2>0$} in direction $\xi \in \mathcal{O} \cap S^{d-1}$ if there exist
\begin{itemize}
    \item a half-space adapted subgroup $H_0< H$,
    \item  a neighborhood $W'$ of $\xi$ in $S^{d-1}$ and $R' > 0$;
\end{itemize} such that the following holds: For every pair $(W'',R'')$ consisting of a $\xi$-neighborhood $W'' \subseteq W'$ and $R''\geq R'$ there exists a $\xi$-neighborhood $W$ and $R>0$ such that for all $L \ge 0$ and all 
\[
s \ge \gamma_0 + \gamma_1 L~,~ t \ge \gamma_0 + \gamma_2 L
\]
the estimate
\[
\forall h \in K_o(W,V,R) : \int_{H_0 \setminus h^{-1}K_i(W'',V,R'')} \Theta_{s,t} (g) \dd g \le C \| h \|^L
\] holds with a constant $C$ independent of $h \in K_o(W,V,R)$.
\end{defn}

Clearly, if $H$ satisfies the $V$-overlap control condition with $W'$ and $R'$ as in the above definition, then it would satisfy the same condition if we replace $W'$ by
$\widetilde{W'}\subseteq W'$, and $R'$ by $\widetilde{R'}\geq R'$.
%%%%%%%%%%%%%%%%%%%%%%%%%%%%%%%%%%%%%%%%%%%%%%%%%%%%%%%%%%%%%%%%%%%%%%%%%%%%%%%%%%%%%%%%
We next want to  establish that the overlap condition only needs to be verified for one direction. This argument is somewhat akin to the proof of the analogous statement for the cone approximation property \cite[Lemma 4.5]{FeFuVo}.
\begin{lem} \label{lem:overlap_control}
If $H$ fulfills the $V$-overlap control condition with parameters $\gamma_0 \ge 0, \gamma_1,\gamma_2>0$ at $\xi_0 \in \mathcal{O}$, then it fulfills this condition at all directions $\xi_1 \in \mathcal{O} \cap S^{d-1}$ with the same parameters $\gamma_0, \gamma_1,\gamma_2$.
\end{lem}
\begin{proof}
Given $\xi_1 \in \mathcal{O} \cap S^{d-1}$, we pick $h \in H$ with $h^T \xi_0 = \xi_1$. 
Assume that $W_0'$ and $R_0'$ are given to guarantee the $V$-overlap control condition at $\xi_0$. 
By appealing to Lemma~\ref{lem:cone_inclusion}\ref{lem:cone_inclusion-2} we find $W'_1,R'_1$ so that 
$K_o(W'_1,V,R'_1) \subseteq h^{-1} K_o(W'_0,V,R'_0)$. Next, we show that $W'_1$ and $R'_1$
satisfy the requirements of the $V$-overlap control condition with parameters $\gamma_1,\gamma_2>0$ in direction $\xi_1$.

Let $W_1'' \subset W_1'$ and $R_1'' > R_1'$ be given. Applying Lemma~\ref{lem:cone_inclusion}\ref{lem:cone_inclusion-1}, we can then find a $\xi_0$-neighborhood $\widetilde{W_0}$ and $\widetilde{R_0}>0$ with  
\[
K_i(W_1'',V,R_1'') \supseteq h^{-1} K_i(\widetilde{W_0},V,\widetilde{R_0})~.
\]
Then $W_0'' = W_0' \cap \widetilde{W_0} \subset W_0'$ and $R_0'' = \max(R_0',\widetilde{R_0}) \ge R_0'$ fulfill 
\[
h^{-1}K_i(W_0',V,R_0') \supseteq h^{-1} K_i(W_0'',V,R_0'')~,
\] as well. 
By $V$-overlap control condition at $\xi_0$, given $W_0'',R_0''$, there exist a $\xi_0$-neighborhood $W_0$ and $R_0>0$ such that for all $L\geq 0$ and all 
$s\geq \gamma_0+\gamma_1 L$ and $t\geq \gamma_0+\gamma_2 L$ the estimate 
\begin{equation}\label{eq:new3}
    \forall\, l \in K_o(W_0,V,R_0) : \int_{H_0 \setminus l^{-1}K_i(W_0'',V,R_0'')} \Theta_{s,t} (g) \dd g \le C \| l \|^L
\end{equation}
holds with a constant $C$ independent of $l \in K_o(W_0,V,R_0)$.
Applying Lemma~\ref{lem:cone_inclusion}\ref{lem:cone_inclusion-2} again, there exists a $\xi_1$-neighborhood $W_1$ and $R_1>0$ such that $ K_o(W_1,V,R_1)\subseteq h^{-1} K_o(W_0,V,R_0)$.

Given $L$, pick $s \ge \gamma_0 + \gamma_1 L$ and $t \ge \gamma_0 + \gamma_2 L$. Let $h_1 \in K_o(W_1,V,R_1)$ be arbitrary, and denote $h h_1=h_0$. 
Now $h^{-1}K_i(W_0'',V,R_0'')\subseteq K_i(W_1'',V,R_1'')$ implies 
$h_0^{-1} K_i(W_0'',V,R_0'') \subseteq h_1^{-1} K_i(W_1'',V,R_1'')$, and thus 
\begin{equation}
\int_{H_0\setminus h_1^{-1}K_i(W_1'',V,R_1'')} \Theta_{s,t}(g) \dd g 
\le \int_{H_0 \setminus  h_0^{-1} K_i(W_0'',V,R_0'')} \Theta_{s,t} (g) \dd g ~. \label{eq:new2}
\end{equation}

Clearly, $h_0\in K_o(W_0,V,R_0)$, as $h_1\in K_o(W_1,V, R_1)$. So combining \eqref{eq:new3} and \eqref{eq:new2}, we get 
$$\int_{H_0\setminus h_1^{-1} K_i(W_1'',V,R_1'')} \Theta_{s,t}(g) \dd g \le  C \| h_0 \|^{L} 
\le \underbrace{C \| h \|^L }_{= C'} \| h_1 \|^{L}~,$$
which shows the desired property. 
\end{proof}

The lemma shows that the $V$-overlap condition is a global property on $\mathcal{O} \cap S^{d-1}$, once it is established for a single direction. Accordingly, we can drop the reference to the direction $\xi$ in the subsequent usage of that condition. 
%%%%%%%%%%%%%%%%%%%%%%%%%%%%%%%%%%%%%%%%%%%%%%%%%%%%%%%%%%%%%%%%%%%%%%%%%%%%%%%%%%%%%%
 
 \begin{remark} \label{rem:context_van_moment}
 The following theorem is central to our paper, as it weakens the requirements on the analyzing wavelet from compact support in frequency to finite numbers of vanishing moments. 
 
 But besides this rather direct benefit, the theorem is also of independent interest for generalized wavelet analysis, since it expresses that, as soon as the analyzing wavelet is chosen from a suitable class, 
 a certain type of wavelet coefficient decay is independent of the precise choice of wavelet. In other words, this decay is recognized as an intrinsic property of the analyzed function. This is akin to a  fundamental observation from coorbit theory \cite{FeiGr0,FeiGr1,FeiGr2}. Recall that coorbit spaces quantify integrability properties of wavelet coefficients, and thereby provide a {\em global measure} of wavelet coefficient decay. An important foundational aspect of this theory is {\em consistency}, i.e.~the fact that (up to equivalence) this global decay behaviour is independent of the analyzing wavelet, see \cite{FeiGr0}. The more recent paper \cite{FuRe} provides guarantees for this independence in terms of vanishing moment conditions on wavelets, for the wavelet transforms of the type studied in our paper. 
 
 In this context, Theorem \ref{thm:transfer_decay} can be seen as an analog to the coorbit space results in \cite{FuRe}, by guaranteeing a similar type of consistency of wavelet coefficient decay localized inside cone-affiliated sets, which holds for all wavelets with sufficiently many vanishing moments. 
 \end{remark}
 
 \begin{thm} \label{thm:transfer_decay}
 Let $x_0\in \RR^d$ and $\xi_0\in \mathcal{O}\cap S^{d-1}$ be fixed.
Assume that the dual action is $V$-microlocally
admissible in direction $\xi_0$, with associated $\xi_0$-neighborhood $W_0$ and $R_0>0$, and with exponent $\alpha_1$. 
Let $N\in {\mathbb N}$ be fixed.

Let $u\in{\mathcal S}'(\RR^d)$ be a tempered distribution of order $N(u)$. 
Let $\psi_1,\psi_2\in {\mathcal S}({\mathbb R}^d)$ be  wavelets of  vanishing moments of order $r$  in $\mathcal{O}^c$, with $\psi_1$ being real-valued.
Assume that the parameters $\beta_1,\beta_2,\beta_3$ and $r$ fulfill the following inequalities: 
 \begin{itemize}
 \item[(i)] $\beta_1\geq N+N(u)+\alpha_1(N(u)+d/2)+d+1$, 
 \item[(ii)] $\beta_2\geq \dim H+N+d+1$, and $\beta_3\geq \dim H+2d$ (which implies $\Theta_{\beta_2-{N}, \beta_3-d}\in \LL^1(H)$),
 \item[(iii)] $\beta_1 \geq 2+ \frac{3}{2}d+N(u)$, $\beta_2 \geq 2+ \dim(H)+\frac{3}{2}d+N(u)$ and $\beta_3 \geq 1 + \dim(H) + 3d  + N(u)$, (which implies  $|\det(\cdot)|^{-p}\Theta_{\beta_2-N(u),\beta_3-N(u)}\in\LL^1(H)$ for {$p=\frac{1}{2}, \frac{3}{2}$}),
 \item[(iv)] $H$ fulfills the $V$-overlap control condition with parameters $\gamma_0\geq 0,\gamma_1,\gamma_2>0$, and 
 \begin{eqnarray*}
 \beta_2 & \ge &  \gamma_0 + \gamma_1 \alpha_1 \left(\frac{3}{2} d+N(u)\right) + N(u) + \gamma_1 N \\
 \beta_3 & \ge & \gamma_0 + \gamma_2 \alpha_1 \left(\frac{3}{2} d+N(u) \right) +N(u) + \frac{3}{2} d + \gamma_2 N
 \end{eqnarray*}
 \item[(v)] $r>d/2+\ell_1({\beta_1+\beta_2}+\beta_3)+\beta_1$, where $\ell_1 >0$ satisfies (\ref{assump}) of Proposition \ref{prop:decay-cross-kernel}.
 \end{itemize}

Let $\epsilon>0$ be given, and assume that there exist suitable choices of $\xi_0$-neighborhood $W',R'$ with $|W_{\psi_1}u(y,h)|\leq C\|h\|^N$ for every $y\in B_\epsilon(x_0)$ and every $h\in K_i(W',V,R')$, where the constant $C$ is independent of $h,y$.
 Then there exist $\xi_0$-neighborhood $W,R$ such that $|W_{\psi_2}u(y,h)|\leq C'\|h\|^N$ for every $y\in B_{\epsilon/2}(x_0)$ and every $h\in K_{o}(W,V,R)$, where the constant $C'$ is independent of the choice of $h,y$.
 \end{thm}
 \begin{proof}
Throughout this proof note that assumption $(v)$ guarantees that Proposition \ref{prop:decay-cross-kernel} is available to estimate $\WW_{\psi_1} \psi_2$. 
 
 Suppose that $W',R'$ are as provided by the premise of the theorem, that is, 
 \begin{equation}\label{eq:premis}
     |\WW_{\psi_1}u(y,g)|\leq C\|g\|^N \ \mbox{ for every } y\in B_\epsilon(x_0) \ \mbox{ and every } g\in K_i(W',V,R'),
 \end{equation} 
 where the constant $C$ is independent of $y$ and $g$. 
 We first note that the premise remains true if we replace $W'$ by sufficiently small $W'' \subset W'$ and $R'$ by sufficiently large $R''>R'$ to make the overlap control condition applicable, and we choose $W,R$ as provided by that condition so that for $s,t$ as in Definition~\ref{defn:overlap_control}, we have
 \begin{equation}\label{eq:control-eq}
\exists C_1>0\  \forall h \in K_o(W,V,R) : \int_{H_0 \setminus h^{-1}K_i(W'',V,R'')} \Theta_{s,t} (g) \dd g \le C_1 \| h \|^L.
\end{equation}
Since replacing $W$ and $R$ by $W\cap W_0$ and $\max\{R,R_0\}$ does not disturb the conclusion of the overlap control condition, we can also assume that  $W\subseteq W_0$ and $R\geq R_0$.

We start out by verifying that, under the assumptions of the theorem, the convolution relation 
\[ W_{\psi_2} u=\frac{2}{C_{\psi_1}}(\WW_{\psi_1}u*\WW_{\psi_2}\psi_1) 
\] holds with convolution taken over $\mathbb{R}^d \rtimes H_0$. This is done most easily by comparing the assumptions (iii) and (v) with those of Proposition~\ref{prop:5.3}, and with (\ref{eqn:beta_cond_conv}) and (\ref{eqn:beta_cond_conv2}).

Throughout the following computations, let $h \in K_o(W,V,R)$, and denote $K = K_i(W'',V,R'')$.
We now split the convolution integral into $\WW_{\psi_2}u(y,h)=\frac{2}{C_{\psi_1}}(I_1+I_2)$, where
 \begin{eqnarray*}
 I_1&=&\int_{h^{-1}B_{\epsilon/2}(0)}\int_{H_0} \WW_{\psi_1}u(y+hx,hg)\WW_{\psi_2}\psi_1(-g^{-1}x,g^{-1})\dd x\frac{\dd g}{|\det(g)|},\\
I_2 &=&\int_{\RR^d\setminus h^{-1}B_{\epsilon/2}(0)}\int_{H_0}\WW_{\psi_1}u(y+hx,hg)\WW_{\psi_2}\psi_1(-g^{-1}x,g^{-1}) \dd x\frac{\dd g}{|\det(g)|}.
 \end{eqnarray*}
Let $y\in B_{\epsilon/2}(x_0)$ and  $h\in K_o(W,V,R)$ be given. We prove appropriate upper bounds for $I_1$ and $I_2$ separately.
 \begin{eqnarray*}
 |I_1|&=&\left|\int_{h^{-1}B_{\epsilon/2}(0)}\int_{H_0} \WW_{\psi_1}u(y+hx,hg)\WW_{\psi_2}\psi_1(-g^{-1}x,g^{-1})\dd x\frac{\dd g}{|\det(g)|}\right|\\
&\leq&\int_{h^{-1}B_{\epsilon/2}(0)}\int_{H_0} \left|\WW_{\psi_1}u(y+hx,hg)\WW_{\psi_1}\psi_2(x,g)\right|\dd x\frac{\dd g}{|\det(g)|}\\
&=&I_1'+I_1'',
\end{eqnarray*}
where we split the integral into $I'_1$ and $I''_1$ by taking the inner integral over $h^{-1}K$ and $H_0\setminus h^{-1}K$ respectively. The point of this separation of integrals is that for all $(x,g) \in h^{-1}(B_{\epsilon/2}(0)) \times h^{-1}K$, one has $(y+hx,hg) \in B_{\epsilon}(x_0) \times K$, and thus the decay assumptions on $W_{\psi_1} u$ can be used to continue
\begin{eqnarray*}
I'_1&=&\int_{h^{-1}B_{\epsilon/2}(0)}\int_{h^{-1}K} \left|\WW_{\psi_1}u(y+hx,hg)\WW_{\psi_1}\psi_2(x,g)\right|\dd x\frac{\dd g}{|\det(g)|}\\
&\leq&C\|h\|^N\int_{h^{-1}B_{\epsilon/2}(0)}\int_{h^{-1}K}{\|g\|^N}|\WW_{\psi_1}\psi_2(x,g)|\dd x\frac{\dd g}{|\det(g)|}\\
&\leq&CD\|h\|^N\int_{h^{-1}B_{\epsilon/2}(0)}\int_{h^{-1}K} (1+|x|)^{-\beta_1} (1+\|g\|)^{{N}-\beta_2}(1+\|g^{-1}\|)^{-\beta_3}\dd x\frac{\dd g}{|\det(g)|}\\
&\leq&CD\|h\|^N\left(\int_{h^{-1}B_{\epsilon/2}(0)}(1+|x|)^{-\beta_1} \dd x\right)\left(\int_{h^{-1}K}(1+\|g\|)^{{N}-\beta_2}{\|g^{-1}\|^d}(1+\|g^{-1}\|)^{-\beta_3} {\dd g}\right)\\
&\leq&CD\|h\|^N\left(\int_{\RR^d}(1+|x|)^{-\beta_1} \dd x\right)\left(\int_{H}(1+\|g\|)^{-(\beta_2-{N})}(1+\|g^{-1}\|)^{-(\beta_3-d)}{\dd g}\right)\\
%&\leq&C|\det(h)|^{-1}\|h\|^N\|\Theta_{\alpha_2,\alpha_3}\|_{\LL^1(H)}\int_{B_\epsilon(0)}(1+|x|)^{-\alpha_1} \dd x\\
&\leq&D_1\|h\|^N,
 \end{eqnarray*}
 where the two integrals in the penultimate line are bounded, since $\beta_1> d$ and $\Theta_{\beta_2-{N}, \beta_3-d}$ belongs to $\LL^1(H)$ thanks to condition (ii). 
% \marginpar{\tiny{MG: Is this correct?}}
 Also note that $D_1$ is a product of the constant from Proposition \ref{prop:decay-cross-kernel}, the constant $C$ from the assumption, and terms depending on $\beta_1,\beta_2$, and $\beta_3$.
 
 To bound $I_1''$, we use (b) of Lemma \ref{lem:wc_decay_general} and Proposition~\ref{prop:decay-cross-kernel}. To avoid clutter of notation, we will denote $N(u)$ by $N'$.
 %from \ref{enu:WaveletTrafoGeneralEstimateForDistributions} 
 %
 \begin{eqnarray*}
I''_1&=&\int_{h^{-1}B_{\epsilon/2}(0)}\int_{H_0\setminus h^{-1}K} \left|\WW_{\psi_1}u(y+hx,hg)\WW_{\psi_1}\psi_2(x,g)\right|\dd x\frac{\dd g}{|\det(g)|}\\
&\leq& \int_{h^{-1}B_{\epsilon/2}(0)}\int_{H_0\setminus h^{-1}K}C_2|\det(hg)|^{-1/2}(1+\|g^{-1}h^{-1}\|)^{N'}\max\{1, \|hg\|^{N'}\}(1+|y+hx|)^{N'}\\
&&\qquad \qquad\qquad \quad\cdot(1+|x|)^{-\beta_1} (1+\|g\|)^{-\beta_2}(1+\|g^{-1}\|)^{-\beta_3} \dd x\frac{\dd g}{|\det(g)|}\\
&=& C_2J_1J'_1,
 \end{eqnarray*}
 where $C_2$ is the product of the constant terms from part (b) of Lemma \ref{lem:wc_decay_general} and Proposition \ref{prop:decay-cross-kernel}, and 
 \begin{eqnarray*}
J_1&=&\int_{h^{-1}B_{\epsilon/2}(0)} (1+|x|)^{-\beta_1}(1+|y+hx|)^{N'} \dd x,\\
J'_1&=&\int_{H_0\setminus h^{-1}K}  |\det(hg)|^{-1/2}(1+\|g^{-1}h^{-1}\|)^{N'} (1+\| hg \|)^{N'} \\ & & (1+\|g\|)^{-\beta_2}(1+\|g^{-1}\|)^{-\beta_3}\frac{\dd g}{|\det(g)|}.\\
 \end{eqnarray*}
It is easy to obtain a bound for $J_1$. Indeed, we note that $|y+hx|\leq |y|+|hx|\leq |x_0|+\epsilon\leq 2|x_0|$, as $x\in h^{-1}B_{\epsilon/2}(0)$, $y\in B_{\epsilon/2}(x_0)$, and  without loss of generality we can assume $\epsilon\leq |x_0|$. So,
$$J_1\leq (1+2|x_0|)^{N'}|\det(h)|^{-1}\int_{B_{\epsilon/2}(0)} (1+|h^{-1}x|)^{-\beta_1} \dd x\leq (1+2|x_0|)^{N'}|\det(h)|^{-1}\lambda_{\RR^d}({B_{\epsilon/2}(0)}),$$
where $\lambda_{\RR^d}({B_{\epsilon/2}(0)})$ is the volume of the $\frac{\epsilon}{2}$-ball.

We next estimate the various factors in the integrand of $J_1'$ as follows: 
\begin{eqnarray*}
    \frac{|{\rm det}(hg)|^{-1/2}}{|{\rm det}(g)|}  & \le &   |{\rm det}(h)|^{-1/2} (1+\| g^{-1} \|)^{\frac{3}{2}d}~, \\
    (1+\| g^{-1} h^{-1} \|)^{N'} & \le & (1+\| h^{-1} \|)^{N'} (1+\| g^{-1} \|)^{N'} ~,  \\
    (1+\| hg\|)^{N'} & \le & (1+\|h\|)^{N'}(1+\|g\|)^{N'}~.
\end{eqnarray*}
Here we used Hadamard's inequality to estimate the determinant against the norm, and submultiplicativity of the norm. 
These estimates lead to 
\begin{eqnarray*}
{J'_1} & \le &  |\det(h)|^{-1/2} (1+\|h\|)^{N'} (1+\|h^{-1}\|)^{N'} \int_{H_0 \setminus  h^{-1}K} \Theta_{\beta_2-N',\beta_3-N'-\frac{3}{2}  d}(g) dg \\
& \le & {C_1} |\det(h)|^{-1/2} (1+\|h\|)^{N'} (1+\|h^{-1}\|)^{N'} \| h \|^L~.
\end{eqnarray*} Here the second inequality used the overlap control condition, applied with 
\[
{s} = \beta_2 - N'~,~ {t} = \beta_3-N'-\frac{3}{2}d~,~ L = \alpha_1 \left(\frac{3}{2} d+N' \right) + N~.
\] The applicability of this condition is guaranteed by assumption (iv),
as well as $h \in K_o(W,V,R)$, %$K = K_i(W',V,R')\supseteq K_i(W'',V,R'')$ 
$K = K_i(W'',V,R'')$ 
and the choice of $W,R$ depending on $W'',R''$ according to the mentioned condition (see \eqref{eq:control-eq}).

Using Hadamard's inequality, followed by property~\ref{enu:NormOfInverseEstimateOnKo} of microlocal admissibility (Definition~\ref{defn:micro_regular}) for $h\in K_o(W,V,R)$, we have 
$$|\det(h^{-1})|^{3/2}\leq \|h^{-1}\|^{\frac{3}{2}d}\lesssim \|h\|^{-\frac{3}{2}\alpha_1 d},$$
which allows to estimate $I_1''$:
\begin{eqnarray*}
I_1'' & \le&   {C_3} |\det(h)|^{-3/2} (1+\|h\|)^{N'} (1+\|h^{-1}\|)^{N'} \| h \|^{N+\alpha_1(\frac{3}{2} d+N') } \\ 
& \le& {C_4 (1+\|h\|)^{N'} (1+\|h^{-1}\|)^{N'}  \| h \|^{N+\alpha_1 N' }~.}
\end{eqnarray*}
Now, if $\|h^{-1}\|\leq 1$, then $1+\|h^{-1}\|\leq 2$; otherwise, $1+\|h^{-1}\|\leq 2\|h^{-1}\|\lesssim \|h\|^{-\alpha_1}$ by microlocal admissibility. In any case, we obtain
\[
 (1+\|h^{-1}\|)^{N'}  \| h \|^{N+\alpha_1 N'}
 \le C_5 \cdot\max\{\| h \|^N, \| h \|^{N+\alpha_1 N'}\} ~.\]
 Furthermore, Lemma~\ref{lem:related_to_assumptions} supplies that $\|h \|$ is bounded on $K_o(W,V,R)$, whence we finally arrive at 
 \[
I_1'' \le C_6 \| h \|^N~.
\] This also concludes the estimate of $I_1$. 

%
%%%%%%%%%%%%%%%%%%%%%%%%%%%%%%%%%%%%%%%%%%%%%%%%%%%%%%%%%%%%%%% 
 We now estimate $I_2$, as follows:
 \begin{eqnarray*}
|I_2|&\leq& \int_{\RR^d\setminus h^{-1}B_{\epsilon/2}(0)}\int_{H_0} \left|\WW_{\psi_1}u(y+hx,hg)\right|\left|\WW_{\psi_2}\psi_1(-g^{-1}x,g^{-1})\right| \dd x\frac{\dd g}{|\det(g)|}\\
&=& \int_{\RR^d\setminus h^{-1}B_{\epsilon/2}(0)}\int_{H_0} \left|\WW_{\psi_1}u(y+hx,hg)\right|\left|\WW_{\psi_1}\psi_2(x,g)\right| \dd x\frac{\dd g}{|\det(g)|}\\
&\leq&\int_{\RR^d\setminus h^{-1}B_{\epsilon/2}(0)}\int_{H_0} C_2|\det(hg)|^{-1/2}(1+\|g^{-1}h^{-1}\|)^{N'}\max\{1, \|hg\|^{N'}\}(1+|y+hx|)^{N'}\\
&&\qquad \qquad\qquad \quad\cdot(1+|x|)^{-\beta_1} (1+\|g\|)^{-\beta_2}(1+\|g^{-1}\|)^{-\beta_3} \dd x\frac{\dd g}{|\det(g)|}\\
&=& C_2I'_2I''_2,
 \end{eqnarray*}
 where $C_2$ is the product of the constant terms from part (b) of Lemma \ref{lem:wc_decay_general} and Proposition \ref{prop:decay-cross-kernel}, and 
 \begin{eqnarray*}
I'_2&=&\int_{\RR^d\setminus h^{-1}B_{\epsilon/2}(0)} (1+|x|)^{-\beta_1}(1+|y+hx|)^{N'} \dd x,\\
I''_2&=&\int_{H_0}  |\det(hg)|^{-1/2}(1+\|g^{-1}h^{-1}\|)^{N'}\max\{1, \|hg\|^{N'}\}(1+\|g\|)^{-\beta_2}(1+\|g^{-1}\|)^{-\beta_3}\frac{\dd g}{|\det(g)|}.\\
 \end{eqnarray*}
Note that for every $z\in \RR^d$, we have
$$(1+|z|)^{N'}=(1+|h(h^{-1}z)|)^{N'}\leq (1+\|h\||h^{-1}z|)^{N'}\leq \max\{1,\|h\|\}^{N'}(1+|h^{-1}z|)^{N'},$$
and since $\epsilon\leq |x_0|$, for every $x\in{\mathbb R}^d$
$$(1+|y+x|)^{N'}\leq (\frac{3}{2}|x_0|+1)^{N'}(1+|x|)^{N'} \ \mbox{ for }\ y\in B_{\epsilon/2}(x_0).$$
So,  letting $\omega_{d-1}$ denote the surface area of the $(d-1)$-sphere of radius 1, we get
\begin{eqnarray*}
I'_2& = &\int_{\RR^d\setminus h^{-1}B_{\epsilon/2}(0)} (1+|x|)^{-\beta_1}(1+|y+hx|)^{N'} \dd x\\
&=&|\det(h)|^{-1}\int_{\RR^d\setminus B_{\epsilon/2}(0)} (1+|h^{-1}x|)^{-\beta_1}(1+|y+x|)^{N'} \dd x\\
&\leq&(\frac{3}{2}|x_0|+1)^{N'}|\det(h)|^{-1}\int_{\RR^d\setminus B_{\epsilon/2}(0)} (1+|h^{-1}x|)^{-\beta_1}(1+|x|)^{N'} \dd x\\
&\leq&(\frac{3}{2}|x_0|+1)^{N'}|\det(h)|^{-1}\max\{1,\|h\|\}^{N'}\int_{\RR^d\setminus B_{\epsilon/2}(0)} (1+|h^{-1}x|)^{-\beta_1+N'} \dd x\\
&\leq&(\frac{3}{2}|x_0|+1)^{N'}\max\{1,\|h\|\}^{N'}\int_{\RR^d\setminus h^{-1}B_{\epsilon/2}(0)} (1+|x|)^{-\beta_1+N'} \dd x\\
&\leq&(\frac{3}{2}|x_0|+1)^{N'}\max\{1,\|h\|\}^{N'}\int_{\RR^d\setminus B_{\frac{\epsilon}{2\|h\|}}(0)} \frac{1}{(1+|x|)^{\beta_1-N'}} \dd x\\
&=&\omega_{d-1}(\frac{3}{2}|x_0|+1)^{N'}\max\{1,\|h\|\}^{N'}\int_{\frac{\epsilon}{2\|h\|}}^\infty \frac{r^{d-1}}{(1+r)^{\beta_1-N'}} \dd r\\
&\leq&\frac{\omega_{d-1}(\frac{3}{2}|x_0|+1)^{N'}\max\{1,\|h\|\}^{N'}}{{\beta_1-N'-d}}(\frac{\epsilon}{2})^{d+N'-\beta_1}\|h\|^{\beta_1-N'-d}\\
&\leq& C_6\max\{\|h\|^{\beta_1-N'-d},\|h\|^{\beta_1-d}\},\\
 \end{eqnarray*}
where we used the fact that $B_{\frac{\epsilon}{2\|h\|}}(0)\subseteq h^{-1}B_{\epsilon/2}(0)$, and  $\beta_1-N'>d$. 
Here, $C_6$ is a constant that only depends on $x_0, \beta_1, N(u)$, and $d$.

Finally, using the submultiplicativity of $h\mapsto (1+\|h\|)^{N'}$, we have
 \begin{eqnarray*}
I''_2&=&\int_{{H_0}}  |\det(hg)|^{-1/2}(1+\|g^{-1}h^{-1}\|)^{N'}\max\{1, \|hg\|^{N'}\}(1+\|g\|)^{-\beta_2}(1+\|g^{-1}\|)^{-\beta_3}\frac{\dd g}{|\det(g)|}\\
&\leq&\frac{{(1+\|h\|)^{N'}}(1+\|h^{-1}\|)^{N'}}{|\det(h)|^{1/2}}\int_{{H_0}}  |\det(g)|^{-1/2}(1+\|g^{-1}\|)^{N'-\beta_3}(1+\|g\|)^{N'-\beta_2}\frac{\dd g}{|\det(g)|}\\
&\leq&{C_7}\|h\|^{-\alpha_1(N'+d/2)}\||\det(\cdot)|^{{-3/2}}\Theta_{\beta_2-N',\beta_3-N'}\|_{\LL^1(H)},\\
 \end{eqnarray*}
 where in the last inequality, we used parts (b) and (c) of Lemma~\ref{lem:related_to_assumptions}.
 Putting all these together, we get
\begin{eqnarray*}
|I_2|&\leq& D_2\max\{\|h\|^{\beta_1-N'-d-\alpha_1(N'+d/2)},\|h\|^{\beta_1-d-\alpha_1(N'+d/2)}\}\\
\end{eqnarray*}
where $D_2=C_2C_6C_7 \||\det(\cdot)|^{-3/2}\Theta_{\beta_2-N',\beta_3-N'}\|_{\LL^1(H)}$, which is finite thanks to condition (iii). 
Using condition (i) on $\beta_1$ and the fact that $\|h\|$ is bounded on $K_o(W,V,R)$, we obtain $|I_2|\leq D_3\|h\|^N$, where $D_3$ is clearly a uniform bound independent of the choice of $h \in K_o(W,V,R)$ and $x\in B_{\epsilon/2}(0)$. 
\end{proof}

\begin{remark} \label{rem:overlap_control}
With hindsight on the proof of Theorem \ref{thm:transfer_decay}, we can provide some more intuition regarding the overlap control condition. The initial idea of the proof is to split the integration domain in the convolution formula 
\[ W_{\psi_2} u=\frac{1}{C_{\psi_1}}(\WW_{\psi_1}u*\WW_{\psi_2}\psi_1) ~,
\] into the part where the decay assumptions on $\WW_{\psi_1} u$ apply, and its complement. On the latter, we can only combine moderate growth of $\WW_{\psi_1} u$  with rapid decay of $\WW_{\psi_2} \psi_1$ to achieve the desired estimates. This rationale motivated the introduction of the auxiliary functions $\Theta_{s,t}$ and the various related results. The overlap control condition is intended to provide the right type of estimate that makes the general approach work. 

However, in the presence of a half-space adapted subgroup $H_0$ (e.g., in the shearlet case), rapid decay of $\WW_{\psi_2} \psi_1$ will not be enough, as long as we integrate over the whole group. To see this, recall that $\mathcal{O} = (H_0^{T} \xi_0) \cup -(H_0^{T} \xi_0)$, and for simplicity consider a direction $\xi \in H_0^{T} \xi_0$ in the proof of Theorem 5.14. Assuming further that $V,W \subset H_0^T \xi_0$, it follows that $K_{i/o}(W,V,R) \subset H_0$. As a consequence, it turns out that the part of the integration domain where we solely rely on the decay estimates for $W_{\psi_2} \psi_1$ to estimate $W_{\psi_2} u (x,h)$ contains the full coset $H \setminus H_0$, \textit{independent of the choice of } $h \in K_o(W,V,R)$. As a consequence, the contribution of the integral over the coset is a constant. 

Now it is important to remember that the right-hand side of the desired inequality is a positive power of $h$, which becomes arbitrarily small. Hence any nontrivial contribution from the coset $H \setminus H_0$ is essentially fatal to our efforts. In particular, an analog of Definition \ref{defn:overlap_control} with the whole group $H$ replacing $H_0$ would not make sense. This is why we replace integration over $H$ by integration over $H_0$, as in our formulation of the overlap control condition. 

The only apparent alternative to this choice appears to make assumptions guaranteeing that either the wavelets or the signal have Fourier transforms which vanish on a half-space. However, these assumptions are incompatible with compact support of the involved functions. 
\end{remark}

Note that the conclusion of the theorem allows to derive a decay statement for \emph{outer} cone-affiliated sets of the type $K_o(W,V,R)$ from decay assumptions for \emph{inner} cone-affiliated sets of the type $K_i(W',V,R')$. Thus, as stated in the following corollary, Theorem \ref{thm:transfer_decay} allows to close one of the two gaps mentioned in Remark \ref{rem:gap_almost_char}. 
\begin{cor}\label{cor:close-gap}
Let $u\in{\mathcal S}'({\mathbb R}^d)$ and $(x,\xi)\in {\mathbb R}^d\times ({\mathcal O}\cap S^{d-1})$.
Assume that the dual action is $V$-microlocally admissible in direction $\xi$ with parameters $\alpha_1, \alpha_2$. Also assume that $H$ fulfills the $V$-overlap control condition, and there exists  $\ell_1 >0$ satisfying (\ref{assump}) of Proposition \ref{prop:decay-cross-kernel}. Then, we have
\begin{enumerate}[label=(\alph*)]
\item \label{enu:cor-RegularDirectedPointNecessaryCondition}If $(x,\xi)$
is an $N$-regular directed point of $u$, then there exists a neighborhood
$U$ of $x$, some $R>0$ and a $\xi$-neighborhood $W\subset S^{d-1}$
such that for \emph{all} real-valued admissible $\psi\in\mathcal{S}(\mathbb{R}^{d})$
with ${\rm supp}(\widehat{\psi})\subset V$, the following estimate
holds: 
\[
\exists\, C >0\,\forall\, y\in U\,\forall\, h\in K_{o}(W,V,R)~:~|W_{\psi}u(y,h)|\le C \cdot\|h\|^{N-\alpha_1 d/2}\!\!.
\]
\item \label{enu:corRegularDirectedPointSufficientCondition}Let $\psi\in\mathcal{S}(\mathbb{R}^{d})$
be admissible with ${\rm supp}(\widehat{\psi})\subset V$. Assume there exists 
a neighborhood $U$ of $x$,  some $R>0$ and
a $\xi$-neighborhood $W\subset S^{d-1}$ such that 
\[\exists\, C>0\,\forall\, y\in U\,\forall\, h\in K_{o}(W,V,R)~:~|W_{\psi}u(y,h)|\le C \cdot\left\Vert h\right\Vert^{\alpha_1 N+\frac{3}{2} \alpha_1 d + \alpha_2}
\]
Then $(x,\xi)$ is an $N$-regular directed point of $u$. 
\end{enumerate}
\end{cor}
%\marginpar{\tiny{MG: I mentioned this corollary in the intro.}}

\begin{proof}
Since $\psi$ is bandlimited, it satisfies vanishing moments condition for any $r$. So, conditions (i), (ii), (iii) and (v) of Theorem~\ref{thm:transfer_decay} are automatically satisfied
for the case $\psi_1 = \psi_2 = \psi$. Combining Theorem~\ref{thm:transfer_decay} and Theorem~\ref{thm:almost_char} finishes the proof of \ref{enu:cor-RegularDirectedPointNecessaryCondition}. Part \ref{enu:corRegularDirectedPointSufficientCondition} follows directly from Theorem~\ref{thm:almost_char}.
\end{proof}
\begin{remark}
Another condition that allows to close the gap between inner and outer cone conditions is provided by the {\em cone approximation property}, which guarantees inclusions of the type $K_o(W,V,R) \subset K_i(W',V,R')$, for suitable choices of $W,R$ depending on $W',R'$. 

It is currently not clear how the cone approximation property is related to the assumptions underlying Theorem \ref{thm:almost_char_Cc} and its corollary. 
We note that both conditions are fulfilled for our chief class of examples, the shearlet dilation groups covered in Section \ref{sect:shearlet}. 
\end{remark}

We now present a near characterization of unsigned regular directed points in the case where the wavelet is real-valued and has sufficiently many vanishing moments. Note that there exist compactly supported wavelets fulfilling these conditions. Observe also that this theorem uses the same cone-related subsets for both implications.
\begin{thm}
\label{thm:almost_char_Cc}
Fix $(x,\xi)\in\mathbb{R}^{d}\times(\mathcal{O}\cap S^{d-1})$.
Let $V$ be a relatively compact subset of $\OO$, and 
assume that the dual action is $V$-microlocally
admissible in direction $\xi$, with associated parameters $\alpha_1,\alpha_2>0$. Assume that $H$ fulfills the $V$-overlap condition, as well as inequality (\ref{assump}). Then there exist constants $s_0,s_1,s_2 \ge 0$, depending only on $H$ and $V$, such that the following holds: 

Let $u\in\mathcal{S}'(\mathbb{R}^{d})$ denote a tempered distribution of order $N(u)$.
Let $\psi \in \mathcal{S}(\mathbb{R}^d)$ denote a wavelet with vanishing moments in $\mathcal{O}^c$ of order
\[
r > s_0 + s_1 N + s_2 N(u)~. 
\]
\begin{enumerate}[label=(\alph*)]
\item \label{enu:RegularDirectedPointNecessaryCondition-cc} If $(x,\pm \xi)$
is an unsigned $N$-regular directed point of $u$, then there exists a neighborhood
$U$ of $x$, some $R>0$ and a $\xi$-neighborhood $W\subset S^{d-1}$
such that the following estimate
holds: 
\[
\exists\, C >0\,\forall\, y\in U\,\forall\, h\in K_{o}(W,V,R) \cup K_{o}(-W,V,R)~:~|\WW_{\psi}u(y,h)|\le C \cdot\|h\|^{N-\alpha_1 d/2}\!\!.
\]
%For each such $\psi$, we even have 
%\[
%\qquad\qquad \exists\, C>0\,\forall\, y\in U\,\forall\, h\in K_{i}(W,\,\widehat{\psi}^{-1}\left(\mathbb{C}\setminus\left\{ 0\right\} \right),\, R)~:~|W_{\psi}u(y,h)|\le C\cdot\left\Vert h\right\Vert ^{N-\alpha_1 d/2}\!\!.
%\]

\item \label{enu:RegularDirectedPointSufficientCondition} Let the wavelet $\psi$ be real-valued. Assume
that $U$ is a neighborhood of $x$, and that there are $R>0$ and
a $\xi$-neighborhood $W\subset S^{d-1}$ such that 
\[\exists\, C >0\,\forall\, y\in U\,\forall\, h\in K_{o}(W,V,R) \cup K_o(-W,V,R)~:~|\WW_{\psi}u(y,h)|\le C \cdot\left\Vert h\right\Vert^{\alpha_1 N+\frac{3}{2} \alpha_1 d + \alpha_2}
\]
Then $(x,\pm \xi)$ is an unsigned $N$-regular directed point of $u$. 
\end{enumerate}
\end{thm}
\begin{proof}
Clearly, the lower bound imposed on $r$ via assumptions (i)-(v) of Theorem \ref{thm:transfer_decay} can be guaranteed to hold by a lower estimate of the type
\[
r > s_0 + s_1 N + s_2 N(u)~,
\] with suitably chosen $s_0,s_1,s_2$ depending only on $H$ and $V$. 

For the proof of (a), pick a wavelet $\eta \in \mathcal{S}(\mathbb{R}^d)$ that fulfills ${\rm supp}(\widehat{\eta}) \subset V_0 \subset V$, w.l.o.g. $V_0 \cap -V_0 = \emptyset$. The latter assumption guarantees that ${\rm Re}(\eta), {\rm Im}(\eta)$ are both nonzero, with Fourier transforms supported in $V_0 \cup -V_0 \subset \mathcal{O}$. Hence they are admissible wavelets as well, with all moments in $\mathcal{O}^c$ vanishing. 
By Lemma \ref{lem:unsigned_dir_points}, $(x,\pm \xi)$ is an unsigned $N$-regular directed point of both ${\rm Re}(u)$ and ${\rm Im}(u)$. By Theorem \ref{thm:almost_char}, there exist  a $\xi$-neighborhood $W'$ in $\mathcal{O}\cap S^{d-1}$, a neighborhood $U$ of $x$, and $R'>0$ such that 
\begin{equation}\label{eq:from-thm}
|\WW_{\eta} ({\rm Re}u) (y,h)|\le C \cdot\|h\|^{N-\alpha_1 d/2}~,~|\WW_{\eta} ({\rm Im}u) (y,h)|\le C \cdot\|h\|^{N-\alpha_1 d/2}
\end{equation}
holds for $h \in K_{i}(W',V,R') \cup K_i(-W',V,R')$ and $y\in U$.
Here, we used the fact that the dual action is $V$-microlocal admissible in any direction, in particular in the direction of $-\xi$ (see Remark~\ref{rem-local/global-microlocality}).
Writing 
\[
\eta_1 = {\rm Re}(\eta), \eta_2 = {\rm Im}(\eta), 
\] we have that $W_\eta {\rm Re}(u) = W_{\eta_1} {\rm Re}(u) - i W_{\eta_2} {\rm Re}(u)$, with $W_{\eta_1} {\rm Re}(u)$ and $ W_{\eta_2}{\rm Re}(u)$ real-valued. Hence the analog of (\ref{eq:from-thm}) also holds for $W_{\eta_i} {\rm Re}(u)$, for $i=1,2$.

Now Theorem \ref{thm:transfer_decay}, applied to $\eta_1$, allows to choose $W,R$ depending on $W',R'$ in such a way that the same type of wavelet coefficient decay holds for $\WW_\psi {\rm Re}(u)$ inside $K_o(W,V,R) \cup K_o(-W,V,R)$. The analogous reasoning can be applied to  $\WW_\psi {\rm Im}(u)$, possibly at the expense of further restricting  $W,R$. But then $u = {\rm Re}(u) + i {\rm Im}(u)$ implies the same decay for $\WW_\psi u$.

For the proof of part (b) we pick a wavelet $\eta \in \mathcal{S}(\mathbb{R}^d)$ that fulfills ${\rm supp}(\widehat{\eta}) \subset V\subseteq \mathcal{O}$. Then $\eta$ has vanishing moments of all orders on $\mathcal{O}$. Since $\psi$ is real-valued, we can use Theorem \ref{thm:transfer_decay} to transfer this decay to $\WW_\eta(u)$ for all $h \in K_o(W,V,R) \cup K_o(-W,V,R)$. Thus, Theorem \ref{thm:almost_char} provides that both $(x,\xi)$ and $(x,-\xi)$ are $N$-regular directed points of $u$.
%The strategy for part (b) is similar. Using $\eta_1 = {\rm Re}(\eta), \eta_2 = {\rm Im}(\eta)$ as well as $v_1 = {\rm Re}(u), v_2 = {\rm Im}(u)$, we have by construction that $\eta_1,\eta_2$ are well-defined real-valued wavelets with all moments in $\mathcal{O}^c$ vanishing. With a proper choice of $W',R'$, the assumptions in part (b) yields decay of the prescribed order for $\WW_\psi v_k$, for $k=1,2$. Via Theorem \ref{thm:transfer_decay}, we can transfer this decay to $\WW_{\eta_j} v_k$, for arbitrary $j,k \in \{ 1,2 \}$, and all $h \in K_o(W,V,R) \cup K_o(-W,V,R)$. But this implies the decay of $\WW_{\eta} u$ on the same set, and thus Theorem \ref{thm:almost_char} provides that both $\xi$ and $-\xi$ are regular directed points of $u$. }
\end{proof}

Note that the slight detour via ${\rm Re}(u)$ and ${\rm Im}(u)$ is necessitated by the fact that for complex-valued wavelet $\eta$ and distribution $u$, there is no simple estimate of $|W_{{\rm Re}(\psi)} u|$, $|W_{{\rm Im}(\psi)} u|$ against $|W_\psi u|$ available, which would allow the transfer of decay statements from one side to the other. The same observation applies to $|W_\psi ({\rm Re}(u))|$ and $|W_\psi ({\rm Im}(u))|$. 

As an easy application of the theorem, we note an adaptation of the main result from \cite{FeFuVo}, from bandlimited Schwartz wavelets to real Schwartz wavelets with vanishing moments of arbitrary order. 
\begin{cor} \label{cor:char_vm_infty}
Fix $(x,\xi)\in\mathbb{R}^{d}\times(\mathcal{O}\cap S^{d-1})$.
Let $V$ be a relatively compact subset of $\OO$, and 
assume that the dual action is $V$-microlocally
admissible in direction $\xi$. Further assume that $H$ fulfills the $V$-overlap condition, as well as inequality (\ref{assump}). 
Let $u\in\mathcal{S}'(\mathbb{R}^{d})$, and
let $\psi \in \mathcal{S}(\mathbb{R}^d)$ denote a real-valued wavelet with vanishing moments in $\mathcal{O}^c$ of arbitrary order. Then the following are equivalent:
\begin{enumerate}
\item[(a)] $(x,\pm \xi)$ is an unsigned regular directed point of $u$.
\item[(b)] There exists a neighborhood
$U$ of $x$, some $R>0$ and a $\xi$-neighborhood $W\subset S^{d-1}$
such that the following  holds: 
\[
\forall N > 0 \exists\, C >0\,\forall\, y\in U\,\forall\, h\in K_{o}(W,V,R) \cup K_{o}(-W,V,R)~:~|\WW_{\psi}u(y,h)|\le C \cdot\|h\|^{N }\!\!.
\]
\end{enumerate}
\end{cor}

%%%%%%%%%%%%%%%%%%%%%%%%%%%%%%%%%%%%%%%%%%%%%%%%%%%%%%%%%%%%%%%%%%%%%%%%%%%%%%%%%%%%%%%%%%%%%%%%%%
%%%%%%%%%%%%%%%%%%%%%%%%%%%%%%%%%%%%%%%%%%%%%%%%%%%%%%%%%%%%%%%%%%%%%%%%%%%%%%%%%%%%%%%%%%%%%

\section{Shearlet dilation groups}

\label{sect:shearlet}

Given that the lists of (partly implicit) assumptions of the Theorems \ref{thm:almost_char} and \ref{thm:almost_char_Cc} are somewhat long and technical, it is legitimate to ask whether there are examples fulfilling all of them. It is the main purpose of this section to demonstrate that there exists a fairly large example class for which all the assumptions can be verified, with reasonably explicit expressions for the involved constants. 

We let $T(d,\Bbb{R})$ denote the matrix group of upper triangular matrices of size $d$ with ones on the diagonal. Elements of $T(d,\Bbb{R})$ are called {\em unipotent}. $\mathfrak{t}(d,\mathbb{R})$ refers to the Lie algebra of $T(d,\mathbb{R})$, which consists of the upper triangular matrices with zero diagonal.

\begin{defn} \label{defn:sdg}
Let $H < {\rm GL}(d,\Bbb{R})$ denote an irreducibly admissible dilation group. $H$ is called a {\em generalized shearlet dilation group}, if there exist two closed subgroups $S, D < H$ with the following properties:
\begin{enumerate}
\item[(i)] $S$ is a connected abelian Lie subgroup of $T(d,\Bbb{R})$;
\item[(ii)] $D = \{ \exp(r Y) : r \in \Bbb{R} \}$ is a one-parameter group, where $Y$ is a diagonal matrix;
\item[(iii)] Every $h \in H$ can be written uniquely as $h = \pm d s$, with $d \in D$ and $s \in S$.
\end{enumerate}
$S$ is called the {\em shearing subgroup} of $H$, and $D$ is called the {\em scaling subgroup} of $H$.
\end{defn}

\begin{remark}
The two-dimensional shearlet groups are given by 
\[
H_c = \left\{ \pm \begin{pmatrix} a & t \\ 0 & a^c \end{pmatrix} : a>0, t \in \mathbb{R} \right\} < {\rm GL}(2,\mathbb{R})~.  
\] Here $c \in \mathbb{R}$ can be chosen arbitrarily. 
\end{remark}

\begin{remark} \label{rem:concrete_shearlet_groups}
For dimension $d \ge 3$, emphasis of research so far has been focused on {\em standard} and {\em Toeplitz} shearlet groups, which we define below. We remark that with growing dimension $d$, the number of possible choices distinct from these two classes of groups grows considerably.
\begin{enumerate}
\item[(i)]  For $\lambda=(\lambda_2, \ldots, \lambda_{d})\in \mathbb{R}^{d-1}$, we define the \textit{standard shearlet group in $d$ dimensions} $H^\lambda$ as the set
\begin{align*}
\left\{ \epsilon\, \mathrm{diag}\left(a, a^{\lambda_2},\ldots, a^{\lambda_{d}}\right)
\begin{pmatrix}
1 & t_1 & \ldots & t_{d-1} \\
 & 1 & 0\ldots & 0			\\
 &  & \ddots & 	0	\\
 &  & & 1
\end{pmatrix}
|
\begin{array}{l}
a>0, \\
t_i\in \mathbb{R},\\
\epsilon\in \left\{ \pm 1 \right\}
\end{array}
\right\}.
\end{align*}
\item[(ii)] For  $\delta\in \mathbb{R}$, we define the \textit{Toeplitz shearlet group in $d$ dimensions} $H^\delta$ as the set
\begin{align*}
\left\{\epsilon\,  \mathrm{diag}\left(a, a^{1-\delta},\ldots, a^{1-(d-1)\delta}\right)\cdot T(1, t_1, \ldots, t_{d-1})
|
\begin{array}{l}
a>0,\\
t_i\in {\mathbb R},\\
\epsilon\in \{\pm 1\}
\end{array}
\right\},
\end{align*}
where the matrix $T(1, t_1, \ldots, t_{d-1})$ is defined by
\begin{align*}
T(1, t_1, \ldots, t_{d-1}):=
\begin{pmatrix}
1 & t_1 & t_2 & \ldots & t_{d-2} & t_{d-1} \\
  & 1	& t_1 & t_2 & \ldots & t_{d-2}\\
  && \ddots & \ddots & \ddots & \vdots \\
  &&& 1 & t_1 & t_2\\
  &&&& 1 & t_1\\
  &&&&& 1
\end{pmatrix}.
\end{align*}
\end{enumerate}
\end{remark}

Below we will establish that the results of the previous sections, in particular Theorems \ref{thm:almost_char} and \ref{thm:almost_char_Cc}, are applicable to shearlet groups from a large example class, including the above-listed families, provided that the scaling subgroup fulfills an additional condition. Thankfully, the majority of the necessary technical conditions have already been verified in previous papers. The following lemma already guarantees the applicability of Theorem \ref{thm:almost_char}. 

\begin{lem} \label{lem:shear_cone_approx}
Let $H$ denote a shearlet dilation group with scaling subgroup $D = \exp(\mathbb{R} Y)$, where w.l.o.g. $Y = {\rm diag}(1,\lambda_2,\ldots,\lambda_d)$. Let 
\[
\lambda_{\max} = \max \{ \lambda_i : 2 \le i \le d \}~,~\lambda_{\min} = \min \{ \lambda_i : 2 \le i \le d \}~.
\]
Assume that $0 < \lambda_{\min} \le \lambda_{\max} < 1$. Then
\begin{enumerate}
    \item[(a)] There exists an open $\emptyset\neq V \Subset H$ such that $H$ fulfills the $V$-cone approximation property; that is, for every $\xi\in {\mathcal O}\cap S^{d-1}$ we have the following: for all $\xi-$neighborhoods $W\subseteq S^{d-1}$ and all $R>0$, there exists a $\xi$-neighborhood $W'\subseteq S^{d-1}$ and $R'>0$ satisfying $K_o(W',V,R')\subseteq K_i(W,V,R).$ 
    \item[(b)] $H$ is microlocally admissible, with 
    \begin{equation} \label{eqn:alpha_i}
    \alpha_1 = \frac{1}{\lambda_{\min}} ~,~ \alpha_2 = \frac{d}{\lambda_{\min}} + \epsilon~,
    \end{equation}
    where $\epsilon>0$ is arbitrary. 
\end{enumerate}
\end{lem}

\begin{proof}
The cone approximation property is shown in Proposition 30 of \cite{AlDaDeMDeVFu}, and microlocal admissibility is stated in Proposition 32 of that paper. The constants are not explicitly stated in this result, but they can be found in the proof: Part (a) of Definition \ref{defn:micro_regular} is shown to hold with $\alpha_1 = \frac{1}{\lambda_{\min}}$. For the existence of $\alpha_2$, the reader is referred to Lemma 4.7 of \cite{FeFuVo}, which appeals to the existence of $\alpha_1$ and the cone approximation property. For an actual formula for $\alpha_2$, the reader should consult the proofs of \cite[Lemma 2.8 and Lemma 4.7]{FeFuVo} to find the condition $\alpha_2 > d \alpha_1$.
\end{proof}

We next establish the condition allowing the application of the decay estimates from Proposition \ref{prop:decay-cross-kernel}. Before we formulate the statement, we cite a statement concerning the structure of the shearing subgroup and its Lie algebra. 

\begin{lem}[\cite{AlDaDeMDeVFu} Lemma 5. and Lemma 6.]\label{lem:CanonicalBasis}
Let $S$ be the shearing subgroup of a generalized shearlet dilation group $H\subset \mathrm{GL}({\mathbb R}^d)$. Then the following statements hold:
\begin{enumerate}
\item[(a)] There exists a unique basis $X_2,\ldots, X_d$ of the Lie algebra $\mathfrak{s}$ of $S$ with $X_i^Te_1=e_i$ for $2\leq i \leq d$, called the canonical basis of $\mathfrak{s}$.
\item[(b)] We have $S=\{I_d + X |\, X\in \mathfrak{s} \}$.
\item[(c)] Let $\mathfrak{s}_k = {\rm span}\{X_j : j \ge k \}$, for $k \in 2, \ldots, d$.  These are associative matrix algebras satisfying $\mathfrak{s}_k  \mathfrak{s}_\ell \subset \mathfrak{s}_{\max{\{k,\ell\}}+1}$, where we write  $\mathfrak{s}_m = \{ 0 \}$ for $m >d$.
\item[(d)] $H$ is the inner semidirect product of the normal subgroup $S$ with the closed subgroup $D \cup -D$. 
\end{enumerate} 
\end{lem}

Note that since $\mathfrak{s}$ consists of strictly upper triangular matrices, it is nilpotent. We let $n(\mathfrak{s})$ denote the {\em nilpotency degree of $\mathfrak{s}$}, 
i.e. 
\[
n(\mathfrak{s}) = \min \{ k \in \mathbb{N} : X^k = 0 \ \forall X \in \mathfrak{s} \}~.  
\] As will be seen more clearly in our computations below, $n(\mathfrak{s})$ enters quite naturally in computations with inverse matrices in $H$.

\begin{lem} \label{lem:est_AH}
Let $H$ denote a shearlet dilation group with scaling subgroup $D = \exp(\mathbb{R} Y)$, where w.l.o.g. $Y = {\rm diag}(1,\lambda_2,\ldots,\lambda_d)$. Let $\mathfrak{s}$ denote the Lie algebra of the shearing subgroup, and $n(\mathfrak{s})$ its nilpotency degree. Assume that $\lambda_{\min},\lambda_{\max}$, defined in Lemma \ref{lem:shear_cone_approx}, fulfill $0 < \lambda_{\min} \le \lambda_{\max} < 1$. Then 
\[
 \|h^{\pm1}\|A_H(h)^{\ell_1}\leq 1
\]
holds for all $h \in H$, with
\begin{equation} \label{eqn:ell_1}
\ell_1 = n(\mathfrak{s})+1~.
\end{equation}
\end{lem}
\begin{proof}
We refer to \cite{FuRe}, Theorem 4.10. Note that the cited theorem makes no assumptions on $Y$, and gives the more general formula
\[
\ell_1 = n(\mathfrak{s})-1 + 2 \| Y \|~.
\]
Here we restrict attention to the case  $0 < \lambda_{\min} \le \lambda_{\max} < 1$, which entails $\| Y \| = 1$.
\end{proof}

Thus the only remaining obstacle to the applicability of Theorem \ref{thm:almost_char_Cc} is the overlap control condition from Definition \ref{defn:overlap_control}. Unfortunately, this condition requires somewhat lengthy computations, which we now prepare. Recall that we need to estimate the integral 
\[
\int_{H_0 \setminus h^{-1} K_i(W,V,R)} \Theta_{s,t}(g) \dd g~.
\] We will do this using fairly explicit computations in particular coordinates for $H$. With respect to these coordinates, the following challenges present themselves:
\begin{enumerate}
    \item Estimate norms of products and inverses in $H$;
    \item explicitly determine the integration domain $H_0 \setminus h^{-1} K_i(W,V,R)$ for certain $h$ in $H$;
    \item determine the Haar measure $\dd g$. 
\end{enumerate}

We start by introducing the coordinates we will use: Using the canonical basis $X_2,\ldots, X_d$ of $\mathfrak{s}$, we define, for $t \in \mathbb{R}^{d-1}$ and $a\not= 0$, 
\[
h(t,a) = \left( I_d + \sum_{i=2}^d t_i X_i \right)^{-1} {\rm sgn}(a) {\rm diag}(|a|,|a|^{\lambda_2},\ldots,|a|^{\lambda_d}) \in H~.
\] Here $I_d = {\rm diag}(1,\ldots,1)$ denotes the $d \times d$ identity matrix. It is easy to see that $(t,a) \ni \mathbb{R}^{d-1} \times \mathbb{R}^\times \to H$ is a bijection, and that restricting this mapping to $\mathbb{R}^{d-1} \times \{ 1 \}$ yields a bijection onto $S$. The inversion of the first factor is the cause of some awkwardness later on, but it will allow a fairly simple description of cone-affiliated subset $K_{i/o} \subset H$.

Note that we can write 
\begin{equation}\label{eq:2}
I_d+\sum_{i=2}^d t_i X_i^T = 
  \begin{pmatrix}
  1 &      0^T\\
  t  &       I_{d-1} + A(t)^T              
  \end{pmatrix},
\end{equation}
with $A(t)$ being a $(d-1)\times (d-1)$ strictly upper-triangular matrix satisfying
\begin{equation}
  \label{eq:5}
  \lVert A(t)\rVert \leq C |t|
\end{equation}
with a constant $C$ depending only on $H$.

We next describe group operations with respect to these coordinates: 
\begin{lem} \label{lem:sh_group_coords}
Let $H$ denote a shearlet dilation group, with canonical basis $X_2,\ldots,X_d$ for $\mathfrak{s}$.
\begin{enumerate}
    \item[(a)] There exists a bilinear map $M : \mathbb{R}^{d-1} \times \mathbb{R}^{d-1} \to \mathbb{R}^{d-1}$ such that $h(t,1)^{-1} h(t',1)^{-1} = h(t+t'+M(t,t'),1)^{-1}$. For all $i=2,\ldots, d$, the $i$th entry $M(t,t')_i$ only depends on the entries $t_2,\ldots,t_{i-1}, t_2',\ldots,t_{i-1}'$.
    \item[(b)] There exists a map $B: \mathbb{R}^{d-1} \to \mathbb{R}^{d-1}$, polynomial of degree $n(\mathfrak{s})-1$, such that for all $t \in \mathbb{R}^{d-1}$,
    \[
    h(t,1)^{-1} = h(-t+B(t),1)~.
    \] Here $B(t)_i$ only depends on the entries $t_2,\ldots,t_{i-1}$.
    \item[(c)] For all $a > 0$ and $t= (t_2,\ldots,t_d)^T \in \mathbb{R}^{d-1}$, we have 
    \[
    h(0,a) h(t,1) h(0,a)^{-1} = h((a^{1-\lambda_2} t_2,\ldots,a^{1-\lambda_d} t_d),1)
    \]
\end{enumerate}
\end{lem}
\begin{proof}
For part (a), we have by definition
\begin{eqnarray*}
h(t,1)^{-1} h(t',1)^{-1} & = &  \left(I_d + \sum_{i=2}^d t_i X_i\right)  \left(I_d + \sum_{i=2}^d t_i' X_i\right) \\
& = & I_d + \sum_{i=2}^{d} (t_i+t_i') X_i + \sum_{i,j=2}^d t_i t_j' X_i X_j \\
& = & I_d + \sum_{i=2}^{d} (t_i+t_i') X_i + \sum_{i} M(t,t')_i X_i ~,
\end{eqnarray*}
where $M(t,t')_i$ denotes the coefficient of the $X_i$-coordinate of $\sum_{\ell,k=2}^d t_k t_\ell' X_k X_\ell$. Clearly this coordinate depends bilinearly on the $t_k,t_\ell'$, and only on those with $\ell,k<i$, by part (c) of Lemma \ref{lem:CanonicalBasis}.

The fact that $t \mapsto h(t,1)$ induces a bijection between $\mathbb{R}^{d-1}$ and $S$ entails the existence of a unique map $B'$ with $h(t,1)^{-1} = h(B'(t),1)$. Furthermore, every unipotent matrix is inverted by a finite Neumann series, and for elements $X \in \mathfrak{s}$, the number of terms in the series is at most $n(\mathfrak{s})-1$. Now part (a) of the current lemma implies that $B'$ is a polynomial of degree $\le n(\mathfrak{s})-1$. Furthermore, part (c) of Lemma \ref{lem:CanonicalBasis} implies that $B'(t) = -t+B(t)$, and $B$ has the described properties.

For part (c) we first note that by definition of the canonical basis, the vector $t \in \mathbb{R}^{d-1}$ such that $h = h(t,1)^{-1}$ is just given by the first row of $h$. Computing the conjugation action of the diagonal matrix on the row entries then gives the equation
   \[
    h(0,a) h(t,1)^{-1} h(0,a)^{-1} = h((a^{1-\lambda_2} t_2,\ldots,a^{1-\lambda_d} t_d),1)^{-1}~.
    \] Inverting both sides of the equation yields (c). 
\end{proof}

\begin{lem}  \label{lem:sh_gr_Haar}
In the $(t,a)$-coordinates, left Haar measure on $H$ is determined as
\[
\int_H f(h) \dd h = \int_{\mathbb{R}^\times} \int_{\mathbb{R}^{d-1}} f(h(t,a)) \dd t \frac{\dd a}{|a|^{d-{\rm tr}(Y)}}~,
\]
where $\dd t$ is the Lebesgue measure of $\mathbb{R}^{d-1}$ and $\dd a$ is the Haar measure of the multiplicative group $\mathbb{R}^\times$.
\end{lem}
\begin{proof}
It is easy to see from part (a) of the previous lemma that Haar integration on $S$ is given by
\[
\int_{\mathbb{R}^{d-1}} f(h(t,1)^{-1}) \dd t~.
\] Since $S$ is an abelian group, this implies that 
\[
\int_{\mathbb{R}^{d-1}} f(h(t,1)) \dd t 
\] yields the same Haar integral. Furthermore, one can use part (c) of the previous lemma to determine the impact of the conjugation action of $D$ on Haar measure of $S$. Now a standard formula for the Haar integral of semidirect products, i.e. \cite[15.29]{MR551496}, yields the desired result.
\end{proof}

We next provide approximate descriptions of the cone-affiliated sets with respect to the coordinates that we just introduced. 
In order to do this, we first parameterize the open orbit $\mathcal O$ by the global chart
provided by affine coordinates
\[
\Omega: \Bbb{R}^\times\times\Bbb{R}^{d-1} \to \mathcal O\qquad
\Omega(\tau,v)=\tau (1,v^T)^T,
\]
and consider the corresponding diffeomorphism to its image
\[ \omega:\Bbb{R}^{d-1}\to S^{d-1}\cap \mathcal{O}, \qquad \omega(v)=
\frac{(1,v^T)^T}{\sqrt{1+|v|^2}}.
\]

Given $\epsilon>0$, we set
\[
W_{\epsilon}=\{v\in\Bbb{R}^{d-1} : |v| < \epsilon \}=B_\epsilon(0).
\]
Clearly, $\{W_{\epsilon}:\epsilon>0\}$ is a neighbourhood basis of the origin in
$\Bbb{R}^{d-1}$, and so $\{\omega(W_{\epsilon}) :\epsilon>0\}$ is a neighbourhood basis of
$\xi_0=(1,0,\ldots,0)\in S^{d-1}\cap \mathcal O$.  
Furthermore, for fixed $0<\tau_1<\tau_2$ and $\epsilon_0>0$  the set
\[
V= \Omega(\, (\tau_1,\tau_2)\times W_{\epsilon_0}\, )
\]
is an open subset with $V\Subset \mathcal O$.

%\marginpar{\footnotesize{\blue{My numbers are different from previous version.}}}
\begin{lem} \label{lem:cone_t_a}
 Assume that $\lambda_{\min},\lambda_{\max}$, defined in Lemma \ref{lem:shear_cone_approx}, fulfill $0 < \lambda_{\min} \le \lambda_{\max} < 1$. Let $V =  \Omega(\, (\tau_1,\tau_2)\times W_{\epsilon_0}\, )$, $R>1$, where $0 < \tau_1 < 1 < \tau_2$ and $W = \omega(W_\epsilon)$, with $\epsilon<1$. Suppose $R$ is sufficiently large, meaning,
\begin{equation}\label{eq:R-suff-large}
R > 2\tau_2 \max\left\{1, (\frac{2\epsilon_0}{\epsilon})^{\frac{1}{1-\lambda_{\max}}}, (2C \epsilon_0)^{\frac{1}{1-\lambda_{\max}}} \right\},
\end{equation}
with $C$ from \eqref{eq:5}.
Then, the following inclusions hold: 
\[ \left\{ h(t,a) : |t| < \delta_0, 0 < a < a_0 \right\} \subseteq K_i(W,V,R) \subseteq K_o(W,V,R) \subseteq \left\{ h(t,a) : |t| < \delta_1, 0 < a < a_1 \right\} ~,
\]
where $a_0,a_1,\delta_0,\delta_1$ are given by 
\[
a_0 = \frac{\tau_1}{R}~,~\delta_0 = \frac{\epsilon}{3}
\]
and 
\[
a_1 = \frac{2\tau_2}{R} ~,~ \delta_1 = 3 \epsilon~.
\]
\end{lem}

\begin{proof}
Our argument is inspired by the proof of Proposition 30 in \cite{AlDaDeMDeVFu}, so we will not expand on its details here. 

Given the definition of $V$, we have 
\begin{equation}
\label{eq:V-transformed}
    (h^{-1})^T(V)=\Omega\left((a^{-1}\tau_1,a^{-1}\tau_2)\times (t+(I_{d-1}+A(t)^T)W_{\epsilon_0}^a)\right),
\end{equation}
where $W_{\epsilon_0}^a:={\rm diag}(a^{1-\lambda_2},\ldots,a^{1-\lambda_d})(W_{\epsilon_0})$.
It is also easy to see that for every $0<R$ and $0<\epsilon\leq 1$, we have 
\begin{equation}
\label{eq:cone-omega}
 (R,\infty)\times W_\epsilon\subseteq     \Omega^{-1}\left(C(\omega(W_\epsilon), R)\right)\subseteq (\frac{R}{2},\infty)\times W_\epsilon.
\end{equation}
By \eqref{eq:V-transformed} and \eqref{eq:cone-omega}, if $h=h(t,a)$ belongs to $K_o(W,V,R)$,  then we must have 
\begin{eqnarray*}
&\mbox{(i)}& 0<a<\frac{2\tau_2}{R}\\
&\mbox{(ii)}& \exists \xi\in B_{\epsilon_0}(0) \mbox{ s.t. } |t+(I_{d-1}+A(t)^T){\rm diag}(a^{1-\lambda_2},\ldots,a^{1-\lambda_d})\xi|<\epsilon.
\end{eqnarray*}
Since $R>2\tau_2$, we have $a<1$, and consequently, we get
$$|t|<\epsilon+(1+C|t|)a^{1-\lambda_{\max}}\epsilon_0\leq \epsilon+(1+C|t|)(\frac{2\tau_2}{R})^{1-\lambda_{\max}}\epsilon_0.$$
The above inequality, together with $R > \max\left\{2\tau_2 (\frac{2\epsilon_0}{\epsilon})^{\frac{1}{1-\lambda_{\max}}}, 2\tau_2 (2C \epsilon_0)^{\frac{1}{1-\lambda_{\max}}} \right\}$, gives $|t|<3\epsilon$. 

For the other inclusion, note that by \eqref{eq:cone-omega}, $h\in K_i$ is guaranteed if $\Omega^{-1}((h^{-1})^T(V))\subseteq (R,\infty)\times W_\epsilon$. For the latter inclusion to hold, it is enough to have
\begin{eqnarray*}
&\mbox{(i)}& 0<a<\frac{\tau_1}{R},\\
&\mbox{(ii)}&  |t|+(1+C|t|)a^{1-\lambda_{\max}}\epsilon_0<\epsilon.
\end{eqnarray*}
Combining \eqref{eq:R-suff-large} and $0<a<\frac{\tau_1}{R}$, one can verify that (ii) holds whenever $|t|<\frac{\epsilon}{3}$.
\end{proof}

For estimates of the function $\Theta_{r,s}$, we need to estimate norms of products of the type $h(t,a)^{-1}h(t',a')$. This is provided by the next lemma:
\begin{lem}
\label{lem:est_norm_prod}
Let $(t,a),(t',a') \in \mathbb{R}^{d-1} \times \mathbb{R}^+$. Assume
that $\lambda_{\min},\lambda_{\max}$, defined in Lemma \ref{lem:shear_cone_approx}, fulfill $0 < \lambda_{\min} \le \lambda_{\max} < 1$. 
\begin{enumerate}
    \item[(a)] Let $V =  \Omega(\, (\tau_1,\tau_2)\times W_{\epsilon_0}\, )$, $R$ satisfying \eqref{eq:R-suff-large}, where $0 < \tau_1 < 1 < \tau_2$ and $W = \omega(W_\epsilon)$, with $\epsilon<1$. Then there exists a constant $C_1$ such that for all $h = h(t,a) \in K_o(W,V,R)$, the following inequalities hold: 
    \[
   a^{\lambda_{\min}} \le \| h(t,a) \| \le C_1 a^{\lambda_{\min}}
    \]
    \item[(b)] $\| h(t,a)^{-1} h(t',a') \| \ge a^{-1}a'$~.
    \item[(c)] Assume that $a,a'<1$ and $|t|<\delta_1$. There exist constants $C_{\mathfrak{s}},C_M>0$ depending only on $H$ such that 
    \[ \| h(t,a)^{-1} h(t',a') \| \ge a^{-1} (a')^{\lambda_{\max}} \left( -\delta_1 + (1-C_M \delta_1) \left\{\begin{array}{cc}
\frac{|t'|}{C_{\mathfrak{s}}} & |t'| < C_{\mathfrak{s}} \\ \frac{|t'|^{1/(n-1)}}{C_{\mathfrak{s}}^{1/(n-1)}} & |t'|\ge C_{\mathfrak{s}} \end{array}~  \right. \right) ~,
    \]
    where $n=n(\mathfrak{s})$.
\end{enumerate}
\end{lem}

\begin{proof}
For the proof of $(a)$, we note that the diagonal entries of $h(t,a)$ are $a,a^{\lambda_2},\ldots,a^{\lambda_d}$. Since $h\in K_o(W,V,R)$ and $R$ satisfies \eqref{eq:R-suff-large}, by Lemma~\ref{lem:cone_t_a} we know $a< 1$. So,
\[
\| h(t,a) \| \ge a^{\lambda_{\min}}~.
\]
For the upper bound, first note that by Lemma~\ref{lem:cone_t_a}, we have $|t|<3\epsilon$. 
Now, submultiplicativity of the norm yields 
\begin{eqnarray*}
\| h(t,a) \|  & \le&  \left\| \left( I_d + \sum_{i=2}^d t_i X_i \right)^{-1}\right\| 
\left\| {\rm diag}(a,a^\lambda_2,\ldots,a^{\lambda_d}) \right\|  \\
& \le & C_1 a^{\lambda_{\min}}~,
\end{eqnarray*}
where $C_1$ is the maximum attained by the continuous map $t\mapsto \left\|\left( I_d + \sum_{i=2}^d t_i X_i \right)^{-1}\right\|$ on the compact set $\overline{B_{3\epsilon}(0)}$.

For part (b), we simply use that $a^{-1}a'$ is the upper left entry of $h(t,a)^{-1} h(t',a')$.

For part (c), we first compute
\begin{eqnarray*}
h(t,1)^{-1} h(t',1) & = & h(t,1)^{-1} h(-t'+B(t'),1)^{-1} \\
& = & h(t-t'+B(t')+M(t,-t'+B(t')),1)^{-1} \\
& = & I_d + \sum_{i=2}^d t_i'' X_i~,
\end{eqnarray*}
with 
\[
t'' = t-t'+B(t')+M(t,-t'+B(t'))~,
\] where $B$ is the polynomial map from Lemma \ref{lem:sh_group_coords}(b). 
Hence 
\begin{eqnarray*}
h(t,a)^{-1} h(t',a') &  =&  {\rm diag}(a^{-1},a^{-\lambda_2},\ldots,a^{-\lambda_d})
\left(  I_d + \sum_{i=2}^d t_i'' X_i \right) {\rm diag}(a',(a')^{\lambda_2},\ldots,(a')^{\lambda_d})~.
\end{eqnarray*}
The norm of that product is greater or equal to the norm of its first row, which is \[
(a^{-1}a',a^{-1}(a')^{\lambda_2} t_2'',\ldots, a^{-1} (a')^{\lambda_d} t_d'')~,\]  hence 
\[
\left\| h(t,a)^{-1} h(t',a') \right\| \ge a^{-1} (a')^{\lambda_{\max}} | t''|~.
\]
For the estimate of $|t''|$, note that by Lemma~\ref{lem:sh_group_coords}(b), $-t'+B(t')$ is a polynomial of order $n-1$. 
The polynomial has zero constant term because $h(0,1)^{-1} = h(0,1)$, hence we have 
\begin{equation} \label{eqn:est_Bt}
|-t' + B(t')| \le \sum_{k=1}^{n-1} c_k |t'|^k \le C_{\mathfrak{s}} \left\{ \begin{array}{cc}
|t'| & |t'|< 1 \\ |t'|^{n-1} & |t'|\ge 1
\end{array}, \right. 
\end{equation} with positive constants $c_k,C_{\mathfrak{s}}$ depending on the generators $X_2,\ldots,X_d$ of the Lie algebra $\mathfrak{s}$. Since the inversion map on $H$ is its own inverse, we have
$t' = t'-B(t') + B(-t'+B(t'))$. This implies, via (\ref{eqn:est_Bt}), the estimate
\[
|t'| \le C_{\mathfrak{s}} \left\{ \begin{array}{cc}
|-t'+B(t')| & |-t'+B(t')|< 1 \\ |-t'+B(t')|^{n-1} & |-t'+B(t')|\ge 1 \end{array}~ \right. ~. 
\] Distinguishing the cases $|-t'+B(t')|\ge 1$ and $|-t'+B(t')| <1$ allows to derive 
\[
|-t'+B(t')| \ge \min \left\{ \frac{|t'|^{1/(n-1)}}{C_{\mathfrak{s}}^{1/(n-1)}}, \frac{|t'|}{C_{\mathfrak{s}}} \right\} ~, 
\]
or equivalently, 
\[
|-t'+B(t')| \ge \left\{ \begin{array}{cc}
\frac{|t'|}{C_{\mathfrak{s}}} & |t'| < C_{\mathfrak{s}} \\ \frac{|t'|^{1/(n-1)}}{C_{\mathfrak{s}}^{1/(n-1)}} & |t'|\ge C_{\mathfrak{s}} \end{array}~ \right. ~.
\]

It follows, using $|t| < \delta_1$ and the constant $C_M$ bounding the bilinear map $M$, 
\[
|t''| \ge |-t'+B(t')|(1-C_M \delta_1) - \delta_1 \ge -\delta_1 + (1-C_M \delta_1) \left\{\begin{array}{cc}
\frac{|t'|}{C_{\mathfrak{s}}} & |t'| < C_{\mathfrak{s}} \\ \frac{|t'|^{1/(n-1)}}{C_{\mathfrak{s}}^{1/(n-1)}} & |t'|\ge C_{\mathfrak{s}} \end{array}~  \right..
\]

\end{proof}

\begin{lem} \label{lem:sh_overlap_control}
Let $H= DS \cup -DS$ denote a shearlet dilation group, with $0 < \lambda_{\min} \leq \lambda_{\max} < 1$. Let $H_0 = DS$. Then $H$ fulfills the overlap control condition, with 
\[ \gamma_0= 2d+1~,~\gamma_1 = \frac{\lambda_{\min}}{1-\lambda_{\max}}~,~ \gamma_2 =  \frac{\lambda_{\min}\lambda_{\max}}{1-\lambda_{\max}}\]
\end{lem}

\begin{proof}
By Lemma \ref{lem:overlap_control}, it suffices to concentrate on the direction $\xi_0 = (1,0,\ldots,0)^T$. Let $V,W',R'$ be given as in Lemma \ref{lem:cone_t_a}, i.e. 
$V = \Omega((\tau_1,\tau_2) \times W_{\epsilon_0})$, $W' = \omega(W_{\epsilon'})$ with $\epsilon'<1$, and $R'$ satisfying \eqref{eq:R-suff-large}. 
%\marginpar{\footnotesize{\blue{why do we need $\epsilon'<7C_{\mathfrak{s}}$?}}}
In addition, assume that $\epsilon'<7C_{\mathfrak{s}}$, where $C_{\mathfrak{s}}$ denotes the constant from Lemma~\ref{lem:est_norm_prod}(c). 
Let $W'' \subset W'$ and $R''>R'$ be arbitrary, and w.l.o.g. assume that $W'' = \omega(W_{\epsilon''})$ with $\epsilon'' \le \epsilon'$. 

We pick $R=R''$ and $W = \omega(W_{\epsilon})$,  where 
\begin{equation} \label{eqn:choice_eps}
\epsilon < \min\left\{\epsilon'', \frac{1}{6 C_M}, \frac{\epsilon''}{21 C_{\mathfrak{s}}} \right\}~. 
\end{equation}
%\marginpar{\footnotesize{\blue{I don't see where we used $\frac{1}{4}$.}}}

We need to estimate, for $h \in K_o(W,V,R)$ the integral 
\[
\int_{H_0 \setminus K_i(W'',V,R'')} \Theta_{r,s}(h^{-1} g) \dd g~;
\]
this is the same integral as in the definition of the overlap control condition, since $h\in K_o(W,V,R)\subseteq H_0$ and $\dd g$ is just Haar integration. 
Using the parameterization $g = h(t',a')$, the inclusion relations from Lemma \ref{lem:cone_t_a} with $a_0''=\frac{\tau_1}{R''}, \delta_0''=\frac{\epsilon''}{3}$ imply that it is sufficient to estimate
\[
\underbrace{\int_{a' \ge a_0''} \int_{\mathbb{R}^{d-1}} \Theta_{r,s}(h^{-1}h(t',a')) \dd (h(t',a'))}_{I_1} + \underbrace{\int_{a' \le a_0''} \int_{|t'| \ge \delta_0''}\Theta_{r,s}(h^{-1}h(t',a')) \dd(h(t',a'))}_{I_2}
\]
against a suitable power of $\| h \|$, where $h = h(t,a)$, and $|t| < \delta_1, a < a_1$ with $\delta_1=3\epsilon$ and $a_1=\frac{2\tau_2}{R}$. Here $\dd(h(t',a'))$ denotes the left Haar integration over $H$, as computed in Lemma \ref{lem:sh_gr_Haar} 

For the estimate of $I_1$, we observe that $a' \ge a_0''$ implies via Lemma \ref{lem:est_norm_prod}(b) that 
\[
\| h(t,a)^{-1} h(t',a') \| \ge a^{-1}a_0''~, 
\] hence
\[
(1+\| h(t,a)^{-1} h(t',a') \|)^{-1} \le \frac{a}{a_0''}
\]
and thus for every $0< \alpha < r$
\[
\Theta_{r,s}(h(t,a)^{-1}h(t',a')) \le C a^\alpha \Theta_{r-\alpha, s}(h(t,a)^{-1}h(t',a'))~.
\]
Therefore, as soon as $r-\alpha \ge \dim(H) + d + 1 = 2d+1$ and $s \ge \dim(H) + d = 2d$, which guarantees the integrability of $\Theta_{r-\alpha,s}$ via Lemma \ref{lem:theta-decays}, we have that 
\begin{eqnarray*}
I_1 & \le &  C a^\alpha \int_{H}  \Theta_{r-\alpha, s}(h(t,a)^{-1}g)  dg \\
& \le & C' \| h(t,a) \|^{\alpha/\lambda_{\min}} ~,
\end{eqnarray*} with the last inequality due to Lemma \ref{lem:est_norm_prod}(a). 
Here, $C'$ only depends on $\tau_1$, $R''$, $\alpha,r,s$, and not on $a,t$.

For the estimate of $I_2$, note that $ K_o(W,V,R)\subseteq  K_o(W',V,R')$ and conditions of Lemma~\ref{lem:cone_t_a} for $W',V,R'$ are satisfied.
So, the assumption $h(t,a) \in K_o(W,V,R)$ and the integral bound $|t'|\geq \delta_0''$ entail
\[
|t'| \ge \frac{\epsilon''}{3}~,~ |t| \le \delta_1 = 3 \epsilon < \frac{\epsilon''}{7 C_{\mathfrak{s}}} ~.
\]
It follows successively from (\ref{eqn:choice_eps}) that 
\[
1 - C_M \delta_1 \ge \frac{1}{2}~,~ \left(1 - C_M \delta_1 \right) \frac{|t'|}{C_{\mathfrak{s}}} \ge \frac{\epsilon''}{6C_{\mathfrak{s}}}~,
\] and finally, via Lemma \ref{lem:est_norm_prod}(c),
\[
\| h(t,a)^{-1} h(t',a') \| \ge a^{-1}(a')^{\lambda_{\max}} C' 
\]
with a global constant $C' = \frac{\epsilon''}{42 C_{\mathfrak{s}}}$. It follows for $0< \beta < \min(r,s)$ that 
\begin{eqnarray*}
I_2 & \le & (C')^{-\beta}  \int_{a' \le a_0''} \int_{|t'| \ge \delta_0''}a^{\beta} (a')^{-\beta \lambda_{\max}} \Theta_{r-\beta,s}(h^{-1}h(t',a')) \dd(h(t',a')) \\
& = & (C')^{-\beta} a^{\beta(1-\lambda_{\max})}  \int_{a' \le a_0''} \int_{|t'| \ge \delta_0''}a^{\beta \lambda_{\max}} (a')^{-\beta \lambda_{\max}} \Theta_{r-\beta,s}(h^{-1}h(t',a')) \dd(h(t',a'))  \\
& \le & C'' a^{\beta(1-\lambda_{\max})} \int_H \Theta_{r-\beta,s-\beta \lambda_{\max}}(h^{-1} g) \dd g ~,
\end{eqnarray*}
where the last inequality used
\[
|a| |a'|^{-1} \le \| h(t',a')^{-1} h(t,a) \|~,
\] see Lemma~\ref{lem:est_norm_prod}(b), as well as $\Theta_{r-\beta,s-\beta \lambda_{\max}} \ge 0$.
Thus, again by Lemma \ref{lem:theta-decays} and  Lemma~\ref{lem:est_norm_prod}(a),
\[
I_2 \le C''' \| h \|^{\beta(1-\lambda_{\max})/\lambda_{\min}}
\] with a finite constant $C'''$ as soon as $r \ge 2d+1+\beta$ and $s \ge 2d + \beta \lambda_{\max}$. Here, $C'''$ depends only on $\epsilon''$, $C_{\mathfrak{s}}, \beta,r,s,H$, and is independent of $h$.

For a given $L>0$, substituting $\alpha = L \lambda_{\min}$ and $\beta = L \frac{\lambda_{\min}}{1-\lambda_{\max}}$, we find the desired estimate
\[
\int_{H_0 \setminus K_i(W'',V,R'')} \Theta_{r,s}(h^{-1} g) \dd g \le \tilde{C} \| h \|^{L}~, 
\]for all $h \in K_o(W,V,R)$, as soon as
\[
r \ge 2d+1+L \frac{\lambda_{\min}}{1-\lambda_{\max}}~,~s \ge 2d + L  \frac{\lambda_{\min}\lambda_{\max}}{1-\lambda_{\max}}~.
\]

\end{proof}

With all conditions for the applicability of our main results to shearlet dilation groups being checked, the following theorem spells out the resulting version of Theorem \ref{thm:almost_char_Cc} for this class of groups.

%\marginpar{\footnotesize{\blue{Need $\lambda_{\max} < 1$ for Lemma 6.4.}}}
\begin{thm}
\label{thm:almost_char_Cc_sh}
Let $H$ denote a shearlet dilation group. Let $Y$ denote the infinitesimal generator of the scaling subgroup of $H$, with eigenvalues $\lambda_1,\ldots, \lambda_d$, normalized such that $\lambda_1=1$. 
Let $\lambda_{\min} = \min \{ \lambda_2,\ldots,\lambda_d \}$, 
$\lambda_{\max} = \max \{\lambda_2,\ldots,\lambda_d  \}$, and 
assume that $0<\lambda_{\min}\leq \lambda_{\max}<1$ and $\lambda_{\min}+\lambda_{\max} \geq 1$. Let $n_0$ denote the nilpotency degree of the shearing subgroup of $H$.

Define 
\begin{eqnarray} \label{eqn:s0_sh}
s_0 & = & 1 + \left( \frac{3}{2} + \frac{1}{2 \lambda_{\min}} \right) d + (n_0+1) \left( \frac{13}{2} d + (\frac{1}{2\lambda_{\min}}+\frac{3+3\lambda_{\max}}{2-2\lambda_{\max}}) d + 3 \right) \\
s_1 & =  & 1+ (n_0+1) \left( 1 + \frac{\lambda_{\min}}{1-\lambda_{\max}} + \frac{\lambda_{\min}\lambda_{\max}}{1-\lambda_{\max}} \right) \\
s_2 & =&  1 + \frac{1}{\lambda_{\min}} + (n_0+1) \left( 2 + \frac{1}{\lambda_{\min}} + \frac{2}{1-\lambda_{\max}} \right)~. \label{eqn:s2_sh}
\end{eqnarray}

Let $u\in\mathcal{S}'(\mathbb{R}^{d})$ denote a tempered distribution of order $N(u)$.
Let $\psi \in \mathcal{S}(\mathbb{R}^d)$ denote a real-valued wavelet with vanishing moments in $\mathbb{R} \times \{ 0 \}^{d-1}$ of order
\begin{equation} \label{eqn:lower_bd_vm}
r > s_0 + s_1 N + s_2 N(u)~. 
\end{equation}
\begin{enumerate}[label=(\alph*)]
\item \label{enu:RegularDirectedPointNecessaryCondition-cc} If $(x,\pm \xi)$
is an unsigned $N$-regular directed point of $u$, then there exists a neighborhood
$U$ of $x$, some $R>0$ and a $\xi$-neighborhood $W\subset S^{d-1}$
such that the following estimate
holds: 
\[
\exists\, C >0\,\forall\, y\in U\,\forall\, h\in K_{o}(W,V,R) \cup K_o(-W,V,R)~:~|\WW_{\psi}u(y,h)|\le C \cdot\|h\|^{N-\frac{d}{2\lambda_{\min}}}\, . 
\]

\item \label{enu:RegularDirectedPointSufficientCondition} Assume
that $U$ is a neighborhood of $x$, and that there are $R>0$ and
a $\xi$-neighborhood $W\subset S^{d-1}$ such that 
\[\exists\, C >0\,\forall\, y\in U\,\forall\, h\in K_{o}(W,V,R) \cup K_o(-W,V,R)~:~|\WW_{\psi}u(y,h)|\le C \cdot\left\Vert h\right\Vert^{\frac{N}{\lambda_{\min}}+\frac{3d}{2\lambda_{\min}}+\alpha_2}\, .
\]
Then $(x,\pm \xi)$ is an unsigned $N$-regular directed point of $u$. 
\end{enumerate}
\end{thm}

\begin{proof}
We first collect the various constants required for the application of Theorem \ref{thm:almost_char_Cc}, starting with $\dim(H)=d$. Recall from Lemma \ref{lem:shear_cone_approx} that $H$ is microlocally admissible, with constant
\[
\alpha_1 = \frac{1}{\lambda_{\min}} >1~\mbox{ and }  \alpha_2=\frac{d}{\lambda_{\min}} + \epsilon^*, \mbox{ for any } \epsilon^*>0.
\] 
Lemma \ref{lem:est_AH} contributes the constant $\ell_1 = n_0 +1$ for which inequality \eqref{assump} holds. By Lemma \ref{lem:sh_overlap_control}, $H$ fulfills the overlap control condition with constants 
\[ \gamma_0= 2d+1~,~\gamma_1 = \frac{\lambda_{\min}}{1-\lambda_{\max}}~,~ \gamma_2 =  \frac{\lambda_{\min}\lambda_{\max}}{1-\lambda_{\max}}.\]
Clearly, $\gamma_1 \alpha_1 \ge 1$. Moreover, the assumption $\lambda_{\min}+\lambda_{\max}\geq 1$  implies that
\begin{equation} \label{eqn:est_factors}
\gamma_1\geq 1~, \gamma_2 \alpha_1  \ge 1~.
\end{equation}
As a consequence, the sharpest lower bound for $\beta_1$ contained in the conditions (i) through (iv) of Theorem \ref{thm:transfer_decay} is provided by condition (i), whereas the relevant lower bounds for $\beta_2,\beta_3$ are provided by (iv), as a consequence of (\ref{eqn:est_factors}). Thus the conditions of Theorem \ref{thm:transfer_decay} (i) through (iv) are fulfilled by 
\begin{eqnarray*}
\beta_1 & = & 1+ \left( 1+\frac{1}{2\lambda_{\min}} \right) d + N + \left(1+ \frac{1}{\lambda_{\min}}\right) N(u)~ \\
\beta_2 & = & 1 + \left( 2+ \frac{3}{2-2 \lambda_{\max}} \right) d + \frac{\lambda_{\min}}{1-\lambda_{\max}} N + \left( 1+ \frac{1}{1-\lambda_{\max}} \right) N(u) \\
\beta_3 & = & 1 + \left( \frac{7}{2} + \frac{3 \lambda_{\max}}{2-2\lambda_{\max}} \right) d + \frac{\lambda_{\min}\lambda_{\max}}{1-\lambda_{\max}} N + \frac{1}{1-\lambda_{\max}} N(u)~.
\end{eqnarray*}
By Theorem \ref{thm:transfer_decay}, the relevant lower bound for the number of vanishing moments is given by the inequality
\[
r > \frac{d}{2}+\ell_1(\beta_1 + \beta_2 + \beta_3) + \beta_1~.
\] Plugging in the values for $\beta_1,\beta_2$ and $\beta_3$ that we just obtained, and sorting terms, results in the equivalent inequality
\[
r > s_0 + s_1 N + s_2 N(u)~,
\] with $s_0,s_1,s_2$ given by equations (\ref{eqn:s0_sh}) through (\ref{eqn:s2_sh}). 
\end{proof}

%\marginpar{\blue{\footnotesize{If the previous part is correct, we need to update this paragraph.}}}
Note that Theorem \ref{thm:almost_char_Cc} is applicable to the shearlet dilation group $H$ as soon as $\lambda_{\min}+\lambda_{\max} \ge 1$ and $0< \lambda_{\min} \le \lambda_{\max} < 1$  hold. It is worthwhile noting that these conditions are also sufficient to guarantee that $H$ characterizes the wavefront set of infinite order, by Theorem 28 of \cite{AlDaDeMDeVFu}. Furthermore, Corollary \ref{cor:char_vm_infty} also applies under these conditions.

We note that both the standard and the Toeplitz shearlet groups from Remark \ref{rem:concrete_shearlet_groups} can fulfill these conditions, when suitable parameters are chosen. For the standard case this is obvious, for the Toeplitz case any parameter $0< \delta \le 1/d$ will work.

\subsection{Discussion/comparison/future directions for results on generalized shearlets}

\label{subsect:shearlet_comparison}

It is interesting to note that there is essentially only one condition required to make a shearlet dilation group a useful tool for the characterization of the Sobolev wavefront set. In particular, the structure of the shearing subgroup only influences the involved constants. 

Theorem \ref{thm:almost_char_Cc_sh} can be briefly summarized by the statement that there is an easily computed lower bound on the number $r$ of vanishing moments that allows a near characterization of elements of the unsigned wavefront set, valid for any Schwartz wavelet with $r$ vanishing moments in $\mathbb{R} \times \{ 0 \}^{d-1}$. It establishes an affine dependence of the required number of vanishing moments on the degree of smoothness and/or the order of the distribution under consideration. The main information provided by Theorem \ref{thm:almost_char_Cc_sh}, on top of guaranteeing the applicability of Theorem \ref{thm:almost_char_Cc}, is the concrete expressions for the involved constants, and their dependence on the structure of the underlying group. 

Ultimately only very limited information on the group is required.  Note that $s_0,s_1,s_2$ only need to be determined once for each group, and they depend only on the eigenvalue range for the infinitesimal generator of the scaling subgroup and the nilpotency degree $n_0$ of the shearing subgroup. As the dimension $d$ increases, the number of relevant groups grows rather quickly. For example, it is known that for dimensions $d \ge 7$, there exists an uncountably infinite family of shearing subgroups that are not related by conjugation (see e.g. Remark 15 of \cite{AlDaDeMDeVFu}), and we expect that many of these subgroups can be combined with a scaling subgroup fulfilling the eigenvalue condition of Theorem \ref{thm:almost_char_Cc_sh}. These observations emphasize that Theorem \ref{thm:almost_char_Cc_sh} applies to a large variety of groups. 

It is reasonable to expect that the generality and scope of these results comes at the price of suboptimal estimates and constants. Hence for specific choices of the dilation group $H$ Theorem \ref{thm:almost_char_Cc_sh} should rather be seen as a proof of principle, with considerable room for improvement. 

One instance where currently known results indicate such a possibility concerns the gap between the exponents in the decay conditions of the direct and inverse theorems. In the shearlet case the requirement $0< \lambda_{\min}<1$ implies that $\alpha_1 = \frac{1}{\lambda_{\min}}>1$. This entails that the gap between the exponents, namely $N-\alpha_1 d/2$ and $\alpha_1 N + \frac{3}{2} \alpha_1 d + \alpha_2$ grows linearly with $N$. 

This can be contrasted to the results in \cite{Grohs_2011}. The setting considered there is essentially that of a shearlet dilation group in dimension two, with $\lambda_{\min} = 1/2$ and $n_0 = 1$. A detailed comparison of our results with the direct and inverse theorems (i.e., 3.1 and 5.5) of \cite{Grohs_2011} cannot be easily performed, due to the various differences of notations and definitions. In spirit however, the results are very similar, relying on lower bounds for numbers of vanishing moments that are affine functions of the smoothness degree $N$. A closer inspection indicates that the gap between the decay orders in the direct and inverse theorems, which grows linearly with $N$ in our setting, can be kept under control in \cite{Grohs_2011}. It is currently unclear whether this difference in the results can be addressed within our framework by the use of more refined and possibly more specific estimates than the ones that we are currently using. 

In other respects, our results are stronger and/or more widely applicable than the ones in \cite{Grohs_2011}. To begin with, the results in the cited paper only apply to shearlets in dimension two. As a further important difference we note that \cite{Grohs_2011} imposes the additional constraint $u \in \LL^2(\mathbb{R}^d)$. This assumption implies that the order $N(u)$ is bounded by $d/2$, which is the main reason why $N(u)$ systematically features in our assumptions, but not in \cite{Grohs_2011}. As a consequence of its focus on $\LL^2$, \cite{Grohs_2011} can work with a somewhat extended reservoir of admissible wavelets, which explains additional assumptions in the direct theorem of \cite{Grohs_2011} on the decay properties of $\psi$, that are trivially fulfilled by Schwartz functions, and hence do not appear in our results. 

A final difference between our paper and the predecessor papers \cite{KuLa,Grohs_2011} concerns the role of the unsigned wavefront set. While both \cite{KuLa, Grohs_2011} formulate their results for regular directed points, in both sources the definition of such a point actually refers to Fourier decay inside $C \cup -C$, for a suitably chosen cone $C \subset \mathbb{R}^2$ (see beginning of Section 5 in \cite{KuLa}, and \cite[Definition 2.1]{Grohs_2011}). Hence they effectively refer to unsigned regular directed points in the terminology of our Definition \ref{defn:unsigned_reg_dir}, without mentioning this fact explicitly. And this restriction is not coincidental; it results from the fact that the shearlet groups in \cite{KuLa,Grohs_2011} only use positive scalings $a>0$. In the terminology of our paper, this corresponds precisely to replacing $H$ by the half-space adapted subgroup $H_0$. This choice requires a more restrictive admissibility condition being used in \cite{KuLa,Grohs_2011}, which is ultimately responsible for the inability of the wavelet transform to distinguish directions $\xi$ and $-\xi$ in the wavefront set analysis. 

We emphasize that, in contrast to the mentioned sources, our Theorem \ref{thm:almost_char} does apply to regular directed points proper, as do the results of the precursor paper \cite{FeFuVo}.

\section*{Acknowledgements}
This paper was initiated during a visit of the second author to the Department of Mathematics at RWTH Aachen University in 2018. 
The authors thank the department for the support and hospitality during this visit. The second author also acknowledges support from National Science Foundation Grant DMS-1902301 during the preparation of this article.
%\section{To-do-list, updated}

%\begin{enumerate}
% \item Check, proofread, polish.
% \item Some of the proofs are somewhat sketchily reliant on other papers. We should check where this should be expanded. We need to find the right tradeoff between being self-contained and not boring the reader. 
% \item Formulate a version of Theorems \ref{thm:almost_char} and \ref{thm:almost_char_Cc} for the shearlet case, with explicit constants. Discuss comparison with Grohs/Labate/Kutyniok. This should contain the conditions on the wavelets, the role of the unsigned regular directed points, a comment on the matrix norm to measure decay, etc. (Note: The decay estimate in Grohs/Labate/Kutyniok do not use the matrix norm.) 
 
% \item There are also various other comments regarding the different shearlet groups that may be of interest here. E.g., the precise structure of the shearing group does not seem to enter anywhere in the picture. This is interesting, because this structure does affect other types of quantifying wavelet coefficient decay, such as coorbit space norms. 
 
% We should also point out where the notion of half-space adapted subgroup pays off, and that it is actually vital in the case of shearlets. 
% \item Rewrite abstract and introduction, ideally with a readable summary of the main results (with minimal technical details), and a comparison to previous results. We could include an itemized list of novel contributions.
% \end{enumerate}

\bibliographystyle{abbrv}
\bibliography{wfset}

\end{document}